\documentclass[11pt]{amsart}

\usepackage{amssymb,mathrsfs,graphicx,extpfeil}
\usepackage{amsmath,amsfonts,amssymb,amscd,amsthm,bbm}

\usepackage{epsfig}
\usepackage{indentfirst, latexsym, amssymb, enumerate,amsmath,graphicx}
\usepackage{float}
\usepackage{epsfig}
\usepackage{comment}
\usepackage{caption}
\usepackage{subcaption}
\usepackage{extpfeil}
\usepackage{graphicx,colortbl}
\usepackage{epsfig}
\usepackage{caption}
\usepackage{subcaption}
\usepackage{bm}
\usepackage{tikz}
\usepackage{tikz-cd}
\usetikzlibrary{matrix,shapes,arrows,positioning,chains}

\usepackage{hyperref}

\usepackage{amsmath}
\title[Second-order consensus models with a bonding force]{Collective behaviors of second-order nonlinear consensus models with a bonding force}

\author[Ahn]{Hyunjin Ahn}
\address[Hyunjin Ahn]{\newline Research institute of Mathematics\newline Seoul National University, Seoul 08826, Republic of Korea}
\email{yagamelaito@snu.ac.kr}

\author[Byeon]{Junhyeok Byeon}
\address[Junhyeok Byeon]{\newline Department of Mathematical Sciences\newline Seoul National University, Seoul 08826, Republic of Korea}
\email{giugi2486@snu.ac.kr}

\author[Ha]{Seung-Yeal Ha}
\address[Seung-Yeal Ha]{\newline Department of Mathematical Sciences and Research Institute of Mathematics \newline Seoul National University, Seoul 08826, Republic of Korea} 
\email{syha@snu.ac.kr}

\author[Yoon]{Jaeyoung Yoon}
\address[Jaeyoung Yoon]{\newline Department of Mathematical Sciences\newline Seoul National University, Seoul 08826, Republic of Korea}
\email{jyoung924@snu.ac.kr}

\newtheorem{theorem}{Theorem}[section]
\newtheorem{lemma}{Lemma}[section]
\newtheorem{corollary}{Corollary}[section]
\newtheorem{proposition}{Proposition}[section]
\newtheorem{remark}{Remark}[section]

\newtheorem{definition}{Definition}[section]

\newcommand{\bbr}{\mathbb R}

\newcommand{\bbs}{\mathbb S}

\newcommand{\bx}{\mbox{\boldmath $x$}}
\newcommand{\by}{\mbox{\boldmath $y$}}

\newcommand{\bv}{\mbox{\boldmath $v$}}
\newcommand{\bbf}{\mbox{\boldmath $f$}}

\newcommand{\bc}{\mbox{\boldmath $c$}}
\newcommand{\bn}{\mbox{\boldmath $n$}}

\newcommand{\br}{\bold r}

\def\charf {\mbox{{\text 1}\kern-.30em {\text l}}}

\calclayout

\makeatletter
\newcommand{\newparallel}{\mathrel{\mathpalette\new@parallel\relax}}
\newcommand{\new@parallel}[2]{%
  \begingroup
  \sbox\z@{$#1T$}
  \resizebox{!}{\ht\z@}{\raisebox{\depth}{$\m@th#1/\mkern-5mu/$}}%
  \endgroup
}
\makeatother

\begin{document}
	\tikzstyle{block} = [rectangle, draw, 
	text width=15em, text centered, rounded corners, minimum height=3em]
	\tikzstyle{line} = [draw, -latex']

	\date{\today}
	
	
	\thanks{Acknowledgment: The work of H. Ahn was supported by NRF-2022R1C12007321 and the work of S.-Y. Ha was supported by NRF-2020R1A2C3A01003881.}
	
	\begin{abstract}
		We study the collective behaviors of two second-order nonlinear consensus models with a bonding force, namely the Kuramoto model and the Cucker-Smale model with inter-particle bonding force. The proposed models contain feedback control terms which induce collision avoidance and emergent consensus dynamics in a suitable framework. Through the cooperative interplays between feedback controls, initial state configuration tends to an ordered configuration asymptotically under suitable frameworks which are formulated in terms of system parameters and initial configurations. For a two-particle system on the real line, we show that the relative state tends to the preassigned value asymptotically, and we also provide several numerical examples to analyze the possible nonlinear dynamics of the proposed models, and compare them with analytical results. 
 \end{abstract}
	
	\maketitle
	
	\centerline{\date}
	
\tableofcontents


	\section{Introduction} \label{sec:1}
	\setcounter{equation}{0}
	Collective behaviors of many-body systems are ubiquitous in nature, to name a few, aggregation of bacteria \cite{T-B}, flocking of birds \cite{C-Sm,T-T,VZ}, synchronization of pacemaker cells and fireflies \cite{A-B,B-C-M1,B-C-M,B-D-P,B-B,C-S,Ku2,Pe,P-R,St,Wi1} and swarming of fish \cite{D-M2,D-M3,D-M1}, etc. Among them, we are mainly interested in two collective behaviors ``{\it synchronization}" and ``{\it flocking}".  Synchronization denotes an adjustment of rhythms of weakly coupled limit-cycle oscillators, whereas flocking represents a collective behavior in which particles move with a common velocity by using limited environmental information and simple rules. These collective behaviors were extensively studied via the particle models, ``the Kuramoto model" and ``the Cucker-Smale model" in literature. Moreover, for one-dimensional case the aforementioned models can be integrated into the common first-order nonlinear consensus model with different coupling functions (see \cite{H-K-P-Z-0,H-L-R-S,H-P-Z}). In this work, we are interested in the second-order nonlinear consensus models incorporating bonding control. To set up the stage, we first begin with a brief description on the aforementioned models one by one.
	
Consider a finite ensemble of weakly coupled Kuramoto oscillators whose states are represented by the real-valued ``{\it phase function}". In fact, Kuramoto oscillators can be visualized as rotators moving around the unit circle $\bbs^1$. Let $\theta_i = \theta_i(t)$ and $\omega_i = \omega_i(t)$ be the phase and frequency (phase velocity) of the $i$-th Kuramoto oscillator, respectively. Then, the second-order Kuramoto model \cite{H-L-R-S} reads as: for any $t>0$ and $i\in[N]:=\{1,\cdots,N\}$,
\begin{equation} \label{Ku-second}
\begin{cases}
\displaystyle {\dot \theta}_i = \omega_i, \quad t > 0, \quad  i \in [N], \\
\displaystyle {\dot \omega}_i =  \frac{\kappa_0}{N} \sum_{j=1}^{N} \cos(\theta_j - \theta_i)(\omega_j - \omega_i),
\end{cases}
\end{equation}
where $\kappa_0$ is a nonnegative coupling strength. For the special set of constrained initial data, the system \eqref{Ku} is equivalent to the first-order Kuramoto model (see Proposition \ref{L2.1}):
\begin{equation} \label{Ku}
{\dot \theta}_i = \nu_i  + \frac{\kappa_0}{N} \sum_{j=1}^{N} \sin(\theta_j - \theta_i), \quad t > 0, \quad  i \in [N].
\end{equation}
The emergent dynamics of the second-order model \eqref{Ku-second} has been studied in \cite{H-L-R-S} only for a restricted class of initial data, whereas the emergent dynamics of the first-order model \eqref{Ku} has been extensively studied from diverse perspectives, e.g., complete synchronization \cite{B-D-P,C-H-J-K,C-S,D-X,H-K-R}, critical coupling strength \cite{D-B},  uniform mean field limit \cite{H-K-P-Z,La}, gradient flow formulation \cite{V-W}, discretized model \cite{H-S-Y,S}, kinetic Kuramoto model \cite{B-C-M1,B-C-M,C-C-H-K-K}, etc.  
	
Next, we consider a finite ensemble of Cucker-Smale flocking particles whose mechanical states are given by position and velocity. More precisely, let $\bx_i$ and $\bv_i$ be the position and velocity of the $i$-th Cucker-Smale particle in $\bbr^d$. Then, the Cucker-Smale (in short CS) model \cite{C-Sm} is governed by the following dynamical system:
\begin{equation}
\begin{cases} \label{CS}
\displaystyle {\dot \bx}_i =  \bv_i,  \quad t >0, \quad  i \in [N], \\
\displaystyle {\dot \bv}_i  = \frac{\kappa_0}{N}\sum_{j=1}^N \psi(\|\bx_j-\bx_i\|) ( \bv_j - \bv_i ),
\end{cases}
\end{equation}
where $\|\cdot\|$ is the standard $\ell^2$-norm in $\mathbb{R}^d$. Similar to the Kuramoto model \eqref{Ku}, the CS model \eqref{CS} was also extensively studied from diverse perspectives in a half century. We refer to \cite{A-B-F,C-D-H,C-F-R-T,C-H-L,H-Liu,H-T,M-T,Sh}  and references therein. In this paper, we address the following simple questions:
\begin{itemize}
		\item
		(${\mathcal Q}_A$):~Can we design a bonding force for \eqref{Ku-second} and \eqref{CS} which makes ensemble form organized spatial patterns?
		\vspace{.1cm}
		\item
		(${\mathcal Q}_B$):~If so, under what conditions on system parameters and initial configurations, do the Kuramoto and CS ensembles with inter-particle bonding force exhibit collective self-organized  behaviors?
	\end{itemize}
 For the CS model \eqref{CS}, the primitive version of the above questions was already discussed in \cite{P-K-H} for a restricted setting, but we further generalize the model in \cite{P-K-H} by varying parameters which measure asymptotic inter-particle distances. We also refer to \cite{L-W, L-W2, R-L-W} for the related questions on pattern formation.  In contrast, for the second-order Kuramoto model \eqref{Ku-second}, the above questions were not addressed in previous literature. Thus, the purpose of this paper is to study the above two questions in depth. \newline
 
 The main results of this paper is three-fold. First, we propose the second-order Kuramoto model with a bonding force:
	\begin{equation} \label{Ku-SB}
		\begin{cases}
			\displaystyle {\dot \theta}_i = \omega_i, \quad t > 0, \quad  i \in [N], \\
			\displaystyle {\dot \omega}_i =  \frac{1}{N} \sum_{j=1}^{N} \Big [ \kappa_0 \cos(\theta_j - \theta_i) + \kappa_1 \Big] (\omega_j - \omega_i) + \frac{\kappa_2}{N}\sum_{j=1}^{N} \Big[ |\theta_j - \theta_i|- \theta^{\infty}_{ij} \Big] \textup{sgn}(\theta_j - \theta_i), 
		\end{cases}
	\end{equation}
	where  $\kappa_1$ and $\kappa_2$ are nonnegative constants representing the intensities of bonding interactions, and $\theta^{\infty}_{ij}$ denotes with desirable asymptotic phase spacing between the $i$-th and the $j$-th oscilaltors. Here, we assume that the entries of the matrix $[\theta_{ij}^{\infty}]$ satisfy 
	\begin{equation} \label{A-0}
		\theta_{ii}^{\infty} = 0, \quad i \in [N], \qquad \theta_{ij}^{\infty} =  \theta_{ji}^{\infty}, \quad 1 \leq i \neq j \leq N.
	\end{equation}
Note that the R.H.S. of \eqref{Ku-SB} can be discontinuous at the instant in which $\theta_j = \theta_i$, i.e., when the oscillators collide, forcing terms become discontinuous. Thus, we may regard \eqref{Ku-SB} as the system of differential inclusions and use the Filippov theory \cite{Fi-1,Fi-2} for a global well-posedness of a generalized solution. However,  in this work, we stay in the realm of classical smooth solutions whose existence is guaranteed by the standard Cauchy-Lipschitz theory. In this case, whether a finite-time collision occurs or not is a crucial issue which is directly related with the well-posedness. Our first result yields that finite-time collision can be avoided under a suitable framework for initial configurations and system parameters such as coupling strengths and $[\theta^{\infty}_{ij}]$:
	\begin{equation}\label{A-1}
	\inf_{0 \leq t < \infty} \min_{1 \leq i \neq j \leq N} |\theta_j(t)-\theta_i(t)| > 0,\qquad \kappa_0\min_{i\ne j} \Big( \sup_{t\ge 0} \cos |\theta_j(t) - \theta_i(t)| \Big) + \kappa_1 > 0.
	\end{equation}
For the detailed description of framework, we refer to Theorem \ref{T3.1}.  Once the nonexistence of finite-time collisions is guaranteed, then complete synchronization can be directly followed from energy estimate (see Proposition \ref{KEE}). For a configuration $\{ (\theta_i, \omega_i) \}$, we define an energy functional ${\mathcal E}$:
	\begin{equation} \label{A-1-1}
	\mathcal{E}(t):= \frac{1}{2} \sum_{i=1}^{N}|\omega_i(t)|^2 + \frac{\kappa_2}{4N} \sum_{i,j =1}^{N} \Big(|\theta_j(t) -\theta_i(t)|-\theta^{\infty}_{ij} \Big)^2.
	\end{equation}
	Then, it follows from time-evolution estimate for  $\mathcal{E}$ that 
	\[ \int_{0}^{\infty} \sum_{i,j = 1}^{N}  |\omega_j(t) - \omega_i(t)|^2 dt \lesssim N \mathcal{E}(0). \]		
	Finally, we use Barbalat's lemma to derive complete synchronization (see Theorem \ref{T3.2}):
	\[ \lim_{t\to\infty} \max_{1 \leq i, j \leq N} |\dot{\theta}_j(t)-\dot{\theta}_i(t)|= 0. \]
	 Second, we propose the following Cucker-Smale model with a bonding force:
	 \begin{equation}
		\begin{cases} \label{CSB}
			\displaystyle {\dot{\bx}_i} =\bv_i,\quad t>0,\quad  i \in [N],\\
			\displaystyle {\dot{\bv}_i} =\frac{\kappa_0}{N}\sum_{j=1}^{N}\psi(\|\bx_j-\bx_i \|)\left(\bv_j- \bv_i\right) + \frac{\kappa_1}{N}\sum_{\substack{j=1 \\ j \neq  i }}^{N} 
		\Big \langle \bv_j-\bv_i, \frac{\bx_j-\bx_i}{\|\bx_i-\bx_j\|} \Big \rangle \frac{(\bx_j-\bx_i)}{\|\bx_j-\bx_i \|} \\
			\displaystyle \hspace{0.9cm}+ \frac{\kappa_2}{N}\sum_{\substack{j=1 \\ j \neq  i }}^{N} (\|\bx_j-\bx_i\|- d^{\infty}_{ij}) \frac{(\bx_j-\bx_i)}{ \|\bx_j-\bx_i \|},
		\end{cases}
	\end{equation}
	where $\kappa_1$ and $\kappa_2$ are nonnegative constants controlling the intensities of CS interactions and $\langle \cdot, \cdot \rangle$ is the standard inner product in $\bbr^d$. 
	The system parameter matrix ${\mathcal D}^{\infty} = [d_{ij}^{\infty}]$ is assumed to satisfy the symmetry conditions \eqref{A-0}, and the communication weight function $\psi:~{\mathbb R_+} \rightarrow \bbr_+$ is nonnegative, bounded, locally Lipschitz continuous and strictly positive in some neighborhood of $0$: 
	\begin{equation}\label{A-2}
		0\leq\psi(r)\leq \psi_M,\quad  r\geq 0, \quad  \psi_m:=\min_{r \in [0, U]} \psi(r)>0, 
	\end{equation}	
where $U$ is defined in \eqref{D-0-0}.  The case $d_{ij}^{\infty} = 2R$ in \eqref{CSB} was treated in \cite{P-K-H}. Hence our proposed model \eqref{CSB} is slightly different from the model proposed in \cite{P-K-H} (the distinction between two models will be discussed in Remark \ref{R4.1}). 

Since the R.H.S. of \eqref{CSB} contains $\| \bx_j - \bx_i \|$ in denominators, it can be singular at the instant in which $\bx_i = \bx_j$. Therefore, as in \eqref{Ku-SB}, finite-time collision avoidance will be a crucial matter for the global well-posedness of classical solutions. For this, parallel to \eqref{A-1-1}, we can define an energy functional $E$:
	\begin{equation} \label{A-2-1}
		E :=\sum_{i=1}^{N}\frac{\|\bv_i\|^2}{2} + \frac{\kappa_2}{4N}\sum_{i,j =1}^{N} \Big(\|\bx_j-\bx_i\|-d^{\infty}_{ij} \Big)^2.
	\end{equation}
	
	In fact, under the formal correspondence:
	
	\[ (\theta_i, \omega_i)  \quad \Longleftrightarrow \quad (\bx_i, \bv_i), \]
	
	the energy functionals defined in \eqref{A-1-1} and \eqref{A-2-1} take the same form. Now, we return to the collision avoidance issue for \eqref{CSB}. Suppose that initial data and system parameters satisfy 
	
	\[ \min_{1 \leq i,j \leq N} \| \bx_i^0 - \bx_j^0 \| > 0, \quad E(0)<\frac{\kappa_2}{N} \Big( \min_{1 \leq i, j \leq N} d^{\infty}_{ij} \Big)^2. \]
	
	Then, particles do not collide in any finite-time interval and relative distances are uniformly bounded (Theorem \ref{T4.1}):
	\begin{align*}
		0 &< \min_{1 \leq i \neq j \leq N}  \inf_{0 \leq t < \infty} \|\bx_i(t) - \bx_j(t) \| \leq \max_{1 \leq i \neq j \leq N}  \sup_{0 \leq t < \infty} \|\bx_i(t) - \bx_j(t) \| < \infty.
	\end{align*}
	On the other hand, under suitable conditions on $\psi$ in \eqref{CSB}, initial data and system parameters, the mono-cluster flocking emerges asymptotically (see Theorem \ref{T4.2}): 
	\begin{align*}
		\sup_{0 \leq t < \infty} \max_{1\leq i, j\leq N} \|\bx_i(t)-\bx_j(t)\| < \infty,\quad\lim_{t \to \infty}  \max_{1\leq i, j \leq N}  \|\bv_j(t)-\bv_i(t)\| =0.
	\end{align*}
	Third, we provide a global existence of Filippov solution for \eqref{CSB} with $N = 2$ on the real line with the desired convergence estimate:
	\[ \lim_{t \to \infty} |x_1(t) - x_2(t) | = d_{12}^{\infty}. \]
	See Section \ref{sec:4.3} for details. \newline
	
The rest of this paper is organized as follows. In Section \ref{sec:2}, we study several basic estimates for the system \eqref{Ku-SB} and \eqref{CSB}. In Section \ref{sec:3}, we present estimates on finite-time collision avoidance and asymptotic synchronization for the second-order Kuramoto model with a bonding control. In Section \ref{sec:4}, we provide similar estimates on the finite-time collision avoidance and asymptotic flocking estimate for the Cucker-Smale model with a bonding control. Furthermore, for the two-particle system on the real line, we show that spatial relative distance tends to the desired relative distance asymptotically.  In Section \ref{sec:5}, we provide several numerical examples for the proposed models and compare them with analytical results in previous sections. Finally, Section \ref{sec:6} is devoted to a brief summary of main results and some remaining issues for a future work. In Appendix A and B, we provivde heuristic derivations for the bonding feedback controls in the Cucker-Smale and the Kuramoto models, respectively. \newline
	
	\noindent {\bf Notation}: We denote $\Theta, W$ and $X$  by the set of state or state vector, respectively, i.e.,
	\begin{align*}
	\begin{aligned}
       & \Theta := \{ \theta_i \} \quad \mbox{or} \quad (\theta_1, \dots, \theta_N), \quad  W: = \{ \omega_i \} \quad \mbox{or} \quad (\omega_1, \dots, \omega_N), \\
        & X:= \{ \bx_i \} \quad \mbox{or} \quad (\bx_1, \dots, \bx_N), \quad  \quad V:= \{ \bv_i \} \quad \mbox{or} \quad (\bv_1, \dots, \bv_N).
        \end{aligned}
        \end{align*}
For $i,j\in[N]$, we also set 	
		\begin{align*}
		\begin{aligned}
		&\br_{ij} := \bx_i-\bx_j,\quad r_{ij} :=\| \br_{ij}\|, \quad \sum_{i \neq j}^N :=  \sum_{i=1}^{N} \sum_{\substack{j =1 \\ i \neq j}}^N,\quad  \max_{i, j} := \max_{i,j\in[N]},\\
		&\min_{i, j} := \min_{i,j\in[N]},\quad \max_{i \neq j} := \max_{i,j\in[N],i\ne j}, \quad \min_{i \neq j} := \min_{i,j\in[N],i\ne j}.
	\end{aligned}
	\end{align*}
	
	\section{Preliminaries} \label{sec:2}
	\setcounter{equation}{0}
	\hspace{.1cm} In this section, we study a relation between the second-order Kuramoto model and the first-order Kuramoto model, basic a priori estimates on the conservation law and energy estimate. We also discuss parallel issues for the Cucker-Smale model with a bonding control. 	
	\subsection{The second-order Kuramoto model} \label{sec:2.1}
	First, we study the relation between the Cauchy problems to \eqref{Ku-second} and  \eqref{Ku} in the following proposition. 
	\begin{proposition} \label{P2.1}
		Suppose that $ \{ (\theta_i, \omega_i) \}$ is a global smooth solution to \eqref{Ku} with the initial data $\{ (\theta_i^0, \omega_i^0) \}$. Then $\{ (\theta_i, \omega_i = {\dot \theta}_i) \}$ is a global smooth solution to \eqref{Ku-second} with the  constrained initial data $\{( \theta_i^0, \omega_i^0) \}$:
		\[  \omega_i^0 := \nu_i + \frac{\kappa_0}{N} \sum_{j=1}^{N} \sin(\theta^0_j - \theta^0_i), \quad i \in [N]. \]
		Conversely, let $\{ (\theta_i, \omega_i) \}$ be a global smooth solution to \eqref{Ku-second} with the constrained initial data $\{ (\theta_i^0, \omega_i^0) \}$:
		\[ \omega_i^0 = \nu_i + \frac{\kappa_0}{N} \sum_{j=1}^{N} \sin(\theta^0_j - \theta^0_i), \quad i \in [N]. \]
		Then, $\{ \theta_i \}$ is a solution of \eqref{Ku} with the initial data $\{ \theta_i^0 \}$.
	\end{proposition}
	\begin{proof} Note that $\eqref{Ku-second}_1$ can be obtained by diffrentiating $\eqref{Ku}_1$ with respect to $t$. The rest of arguments can be followed in a straightforward manner.
		 \end{proof}
	Next, we study a translation invariance and conserved quantities associated with $\eqref{Ku-SB}$. 
	\begin{lemma} \label{L2.1} Let $\{ (\theta_i, \omega_i) \}$ be a global smooth solution to \eqref{Ku-SB}. Then the following assertions hold.
		\begin{enumerate}
			\item
			System \eqref{Ku-SB} is Galilean invariant, i.e., it is invariant under the Galilean transformation:
			\[ (\theta_i, \omega_i)~ \mapsto ~ (\theta_i + \alpha t, \omega_i + \alpha), \quad \mbox{for}~~ \alpha \in {\mathbb R}, ~~i \in [N]. \] 
			\item
			The total sum of frequencies is preserved:
			\[ \sum_{i=1}^{N} \omega_i(t) =  \sum_{i=1}^{N} \omega_i^0, \quad t > 0. \]
		\end{enumerate}
	\end{lemma}
	\begin{proof}
		(i) For some $\alpha \in {\mathbb R}$, we set 
		\[ \tilde{\theta}_i := \theta_i + \alpha t, \quad \tilde{\omega}_i :=  \omega_i + \alpha. \]
		Then, it is easy to see
		\begin{equation} \label{B-0-0}
			\dot{\tilde{\theta}}_i = \frac{d}{dt} \Big( \theta_i   + \alpha t  \Big) = {\dot \theta}_i   + \alpha = \omega_i + \alpha = {\tilde \omega}_i. 
		\end{equation}
		On the other hand, the R.H.S. of $\eqref{Ku-SB}_2$ is expressed in terms of $\theta_j - \theta_i$ and $\omega_j - \omega_i$. Thus, one has 
		\begin{align} \label{B-0-1}
			\begin{aligned}
				\dot{\tilde \omega}_i &=\frac{\kappa_0}{N} \sum_{j=1}^{N} \cos(\tilde{\theta}_j - \tilde{\theta}_i)({\tilde{\omega}}_j - {\tilde{\omega}}_i)\\
				&+\frac{\kappa_1}{N}\sum_{j=1}^{N}(\tilde{\omega}_j - \tilde{\omega}_i)+ \frac{\kappa_2}{N}\sum_{j=1}^{N}(| \tilde{\theta}_j - \tilde{\theta}_i|-\theta^{\infty}_{ij})\textup{sgn}(\tilde{\theta}_j - \tilde{\theta}_i).
			\end{aligned}
		\end{align}
		Finally, we combine \eqref{B-0-0} and \eqref{B-0-1} to derive the first assertion.  \newline

		\noindent (ii)~ The R.H.S. of \eqref{Ku-SB} is skew-symmetric with respect to index exchange $(i, j) \leftrightarrow (j, i)$. Thus, the total sum $\sum_{i} \omega_i$ satisfies 
		\[ \frac{d}{dt} \sum_{i=1}^{N} \omega(t) = 0, \quad t > 0. \]
		This yields the desired estimate. 
		 \end{proof}
	 Now, we introduce an energy functional associated with \eqref{Ku-SB}. For a given configuration $\{ (\theta_i, \omega_i) \}$, we set 
	\begin{align} \label{B-0-5}
		\begin{aligned}
			& \mathcal{E}:=\mathcal{E}_k+\mathcal{E}_p, \quad  \mathcal{E}_k:=\frac{1}{2} \sum_{i=1}^{N}|\omega_i|^2, \\
			& \mathcal{E}_p:=\frac{\kappa_2}{4N} \sum_{i,j =1}^{N} \Big(|\theta_j-\theta_i|-\theta^{\infty}_{ij} \Big)^2 = \frac{\kappa_2}{2N} \sum_{1 \leq i < j \leq N} \Big(|\theta_j-\theta_i|-\theta^{\infty}_{ij} \Big)^2, 
	\end{aligned}
	\end{align}
	where ${\mathcal E}_k, {\mathcal E}_p$ and ${\mathcal E}$ denote the kinetic energy, the potential energy and the total energy, respectively. 
	\begin{proposition}\label{KEE}
		\emph{(Energy estimate)}
		For $\tau\in(0,\infty]$, let $\{(\theta_i, \omega_i)\}$ be a smooth solution to \eqref{Ku-SB} in the time-interval $t\in[0,\tau)$. Then, the total energy $\mathcal{E}$ satisfies
		\begin{equation} \label{B-0-6}
			\mathcal{E}(t)+\int_{0}^{t} {\mathcal P}(s)ds=\mathcal{E}(0) \quad \mbox{for}\quad t\in[0,\tau),
		\end{equation}
		where the production rate functional ${\mathcal P}(t)$  is given as follows:
		\[ {\mathcal P}(t):=\frac{1}{2N}\sum_{i,j =1}^{N} \Big(\kappa_0 \cos(\theta_j-\theta_i) + \kappa_1 \Big) |\omega_j - \omega_i|^2. \]
		\end{proposition}
	\begin{proof}
		We multiply $\omega_i$ to $\eqref{Ku-SB}_2$ to obtain
		\begin{align}
			\begin{aligned} \label{B-0-7}
				\frac{d}{dt} \frac{|\omega_i|^2}{2} &= \frac{\kappa_0}{N} \sum_{j=1}^{N} \cos(\theta_j - \theta_i)(\omega_j - \omega_i)\omega_i\\
				&+\frac{\kappa_1}{N}\sum_{j=1}^{N}(\omega_j - \omega_i)\omega_i+ \frac{\kappa_2}{N}\sum_{j=1}^{N}(|\theta_j - \theta_i|- \theta^{\infty}_{ij})\textup{sgn}(\theta_j - \theta_i)\omega_i.
			\end{aligned}
		\end{align}
		Then, we sum up \eqref{B-0-7} over all $i \in [N]$, and then use the index interchange trick $i\leftrightarrow~j$ to obtain		
		\begin{align}\label{B-0-8}
			\begin{aligned}
				\frac{d{\mathcal E}_k}{dt} &= \frac{d}{dt}\left(\frac{1}{2}\sum_{i=1}^{N}|\omega_i|^2\right)=-\frac{\kappa_0}{2N}\sum_{i,j =1}^{N}\cos(\theta_j-\theta_i)|\omega_j - \omega_i|^2\\
				&-\frac{\kappa_1}{2N}\sum_{i,j =1}^{N}|\omega_j - \omega_i|^2-\frac{\kappa_2}{2N}\sum_{i,j= 1}^N(|\theta_j - \theta_i|-\theta^{\infty}_{ij})\textup{sgn}(\theta_j - \theta_i)(\omega_j-\omega_i)\\
				&=: -\frac{1}{2N}\sum_{i,j =1}^{N} \Big(\kappa_0 \cos(\theta_j-\theta_i) + \kappa_1 \Big) |\omega_j - \omega_i|^2  + {\mathcal I}_1.
			\end{aligned}
		\end{align}
		Now, we estimate the term ${\mathcal I}_1$ as follows. Since \[ |\theta_j - \theta_i| =  \textup{sgn}(\theta_j - \theta_i) (\theta_j - \theta_i),\]  one has 
		\[ \textup{sgn}(\theta_j - \theta_i)(\omega_j-\omega_i) =  \textup{sgn}(\theta_j - \theta_i)(\dot{\theta}_j-\dot{\theta}_i) = \frac{d}{dt} |\theta_j - \theta_i|. \]
	In fact, the last relation can be made rigorously by approximating $\mbox{sgn}(\cdot)$ by its smooth approximation. We omit its details here. Thus, the term ${\mathcal I}_1$ can be estimated as follows.
		\begin{align}
			\begin{aligned} \label{B-0-9}
				{\mathcal I}_1& = -\frac{\kappa_2}{2N}\sum_{i,j= 1}^N(|\theta_j - \theta_i|-\theta^{\infty}_{ij})\frac{d}{dt}|\theta_j-\theta_i|=-\frac{\kappa_2}{4N}\frac{d}{dt} \Big( \sum_{i,j= 1}^N(|\theta_j - \theta_i|-\theta^{\infty}_{ij})^2 \Big) = -\frac{d{\mathcal E}_p}{dt}.
			\end{aligned}
		\end{align}
		Finally, we combine \eqref{B-0-8} and \eqref{B-0-9} to get the desired energy estimate.
		 \end{proof}
	\subsection{The Cucker-Smale model} \label{sec:2.2}
	\hspace{.1cm} As in Proposition \ref{KEE}, we study the basic structure of \eqref{CSB}.  First, we begin with the invariance and conservation of total momentum to \eqref{CSB}.
	\begin{lemma} \label{L2.2} Let $\{ (\bx_i, \bv_i) \}$ be a solution to \eqref{CSB}. Then the following assertions hold.
		\begin{enumerate}
			\item
			The system \eqref{CSB} is Galilean invariant, i.e., it is invariant under the Galilean transformation:
			\[ (\bx_i, \bv_i)~ \mapsto ~ (\bx_i + \bc t, \bv_i + \bc), \quad \mbox{for $\bc \in {\mathbb R}^d$}, ~i \in [N]. \] 
			\item
			The total sum of velocities is conserved:
			\[ \sum_{i=1}^{N} \bv_i(t) =  \sum_{i=1}^{N} \bv_i^0, \quad t \geq 0. \]
		\end{enumerate}
	\end{lemma}
	\begin{proof} We set 
		\[ {\tilde \bv}_i := \bv_i + \bc, \quad {\tilde \bx}_i := \bx_i + \bc t. \]
		Then, one has 
		\begin{equation} \label{B-0}
			\dot{\tilde \bx} = \dot{\bx}_i + \bc = \bv_i + \bc = {\tilde \bv}_i, \quad i \in [N]. 
		\end{equation}
		On the other hand, note that
	      \[ \bx_j - \bx_i = \tilde{\bx}_j - \tilde{\bx}_i, \quad  \bv_j - \bv_i = \tilde{\bv}_j - \tilde{\bv}_i, \quad i, j \in [N],  \]
		and the R.H.S. of \eqref{CSB} is expressed in terms of relative differences $\bx_j - \bx_i$ and $\bv_j - \bv_i$, and 
		\begin{align}
			\begin{aligned} \label{B-1}
				\dot{\tilde \bv}_i &=\frac{\kappa_0}{N}\sum_{j=1}^{N}\psi({\tilde r}_{ij})\left(\tilde{\bv}_j- \tilde{\bv}_i\right)\\
				&+ \frac{\kappa_1}{N}\sum_{j \neq i}^{N}\frac{\langle \tilde{\bv}_j- \tilde{\bv}_i, \tilde{\bx}_j- \tilde{\bx}_i\rangle}{{\tilde r}^2_{ij}}(\tilde{\bx}_j- \tilde{\bx}_i)+\displaystyle\frac{\kappa_2}{N}\sum_{j \neq i}^{N}\frac{(\tilde{r}_{ij}-d^{\infty}_{ij})}{\tilde{r}_{ij}}(\tilde{\bx}_j- \tilde{\bx}_i),
			\end{aligned}
		\end{align}
		where we defined ${\tilde r}_{ij}:=\| \tilde{\bx}_i- \tilde{\bx}_j\|$. We combine \eqref{B-0} and \eqref{B-1} to get the desired estimate. 
		 \end{proof}
	As in Section \ref{sec:2.1}, we introduce an energy for \eqref{CSB} as follows.
	\begin{align}
	\begin{aligned}  \label{D-1}
		& E:= E_k+ E_p, \quad E_k:=\frac{1}{2}\sum_{i=1}^{N}\|\bv_i\|^2, \\
		& E_p:=\frac{\kappa_2}{4N}\sum_{i,j =1}^{N}(\|\bx_j-\bx_i\|-d^{\infty}_{ij})^2 = \frac{\kappa_2}{2N}\sum_{1 \leq i < j \leq N}(\|\bx_j-\bx_i\|-d^{\infty}_{ij})^2,
		\end{aligned}
	\end{align}
	where $E_k, E_p$ and $E$ correspond to kinetic energy, potential energy and total energy, respectively. 
	\begin{proposition}\label{CSEE}
		\emph{(Energy estimate)}
		For $\tau\in(0,\infty]$, let $\{(\bx_i,\bv_i)\}_{i=1}^N$ be a solution to \eqref{CSB} on $t\in[0,\tau)$. Then, the total energy $\mathcal{E}$ satisfies 
		\[ E(t)+\int_{0}^{t} P(s)ds= E(0) \quad \mbox{for}\quad t\in[0,\tau),\]
		where $P$ is the total energy production functional given by
		\[ P := \frac{\kappa_0}{2N} \sum_{i,j =1}^{N}\psi(\|\bx_j-\bx_i\|)\| \bv_j- \bv_i\|^2+ \frac{\kappa_1}{2N} \sum_{i\ne j}^{N} \Big \langle \bv_j-\bv_i, \frac{\br_{ji}}{r_{ji}}  \Big \rangle^2, \quad t \in (0, \tau).
	\]
	\end{proposition}
	\begin{proof} Below, we estimate the time-derivatives of $E_k$ and $E_p$ one by one.  \newline
		
		\noindent $\bullet$~Case A (Time-derivative of $E_k$):  we use $\eqref{CSB}_2$ and an index interchange trick $i \leftrightarrow j$ to obtain
		\begin{align}
			\begin{aligned} \label{D-3-0}
				\frac{dE_k}{dt} &= \sum_{i=1}^N \langle \bv_i , \dot{\bv}_i \rangle \\
				&= \frac{\kappa_0}{N}\sum_{i,j=1}^{N}\psi(r_{ji})\big\langle  \bv_j- \bv_i,\bv_i\big\rangle +\displaystyle\frac{\kappa_1}{N}\sum_{i \neq j}^{N}\frac{1}{r_{ji}^2}\langle \bv_j-\bv_i,\br_{ji} \rangle\langle\br_{ji},\bv_i\rangle+ \frac{\kappa_2}{N}\sum_{i \neq j}^{N}\frac{r_{ji}-d^{\infty}_{ij}}{r_{ji}}\langle \br_{ji} ,\bv_i\rangle \\
                                &= -\frac{\kappa_0}{2N}\sum_{i,j =1}^{N}\psi(r_{ji}) \| \bv_j- \bv_i \|^2 -\frac{\kappa_1}{2N}\sum_{i,j =1}^{N}\frac{1}{r_{ji}^2}\langle \bv_j-\bv_i,\br_{ji} \rangle^2 - \frac{\kappa_2}{2N}\sum_{i\ne j}^{N}\frac{r_{ji}-d^{\infty}_{ij}}{r_{ji}}\langle \br_{ji} ,\bv_j-\bv_i \rangle.
			\end{aligned}
		\end{align}

\vspace{0.2cm}
		
		\noindent $\bullet$~Case B (Time-derivative of $E_p$): We use $\eqref{CSB}_2$ and  \eqref{D-1} to have
		\begin{equation} \label{D-2}
				\frac{dE_p}{dt} = \frac{\kappa_2}{2N}\sum_{i\ne j}^{N}(r_{ij} -d_{ij}^\infty) \frac{dr_{ij}}{dt}=  \frac{\kappa_2}{2N} \sum_{i\ne j}^{N}\frac{r_{ji}-d_{ji}^\infty}{r_{ji}}\langle \br_{ji},\bv_j-\bv_i\rangle.
		\end{equation}
				Finally, we combine \eqref{D-2} and \eqref{D-3-0} to induce
		\[
				\frac{dE}{dt} =-\frac{\kappa_0}{2N}\sum_{i,j =1}^{N}\psi(r_{ji})\| \bv_j- \bv_i\|^2-\displaystyle\frac{\kappa_1}{2N} \sum_{i\ne j}^{N} \Big \langle \bv_j-\bv_i, \frac{\br_{ji}}{r_{ji}} \Big \rangle^2
				=-P.
	       \]
		This leads to the desired estimate. 
		 \end{proof}

	\section{The Second-order Kruamoto model with a bonding force} \label{sec:3}
	\setcounter{equation}{0}
	\hspace{.1cm} In this section, we provide a sufficient framework which leads to the nonexistence of finite-time collision to \eqref{Ku-SB}, and then we derive the complete synchronization using energy estimate in Proposition \ref{KEE}.  \newline
	
Note that the forcing terms in \eqref{Ku-SB} can be decomposed as the sum of synchronizing force and bonding force:
	\begin{align}
		\begin{aligned} \label{C-1}
			&\frac{1}{N} \sum_{j=1}^{N} \Big [ \kappa_0 \cos(\theta_j - \theta_i) + \kappa_1 \Big] (\omega_j - \omega_i)+ \frac{\kappa_2}{N}\sum_{j=1}^{N} \Big[ |\theta_j - \theta_i|-\theta^{\infty}_{ij} \Big] \textup{sgn}(\theta_j - \theta_i) \\
			&\hspace{0.5cm}=  \underbrace{\frac{\kappa_0}{N} \sum_{j=1}^{N} \cos(\theta_j - \theta_i)(\omega_j - \omega_i)}_{\mbox{synchronization force}}+ \underbrace{\frac{\kappa_1}{N}\sum_{j=1}^{N}(\omega_j - \omega_i)  + \frac{\kappa_2}{N}\sum_{j=1}^{N}(|\theta_j - \theta_i|- \theta^{\infty}_{ij})\textup{sgn}(\theta_j - \theta_i)}_{\mbox{bonding force}}.
		\end{aligned}
	\end{align}
As mentioned in Introduction, the term $\textup{sgn}(\theta_j - \theta_i)$ in the R.H.S. of $\eqref{C-1}$ is discontinuous at the instant when $\theta_i = \theta_j$. Thus, as long as there are no finite-time collisions, the R.H.S of $\eqref{Ku-second}_1$  is still Lipschitz continuous and sublinear with respect to state variable. Therefore, a global well-posedness of classical solutions can be made in a classical framework based on the Cauchy-Lipschitz theory. In what follows,  we are interested in the following two issues:  \newline

	\begin{itemize}
		\item
		Issue A.1:~(Nonexistence of finite-time collisions):~we provide a sufficient framework leading to the nonexistence of finite-time collisions in terms of initial data and system parameters.  
		\vspace{0.1cm}
		\item
		Issue A.2:~(Emergence of complete synchronization):~we present a sufficient framework leading to complete synchronization which represents zero convergence of relative frequencies. 
	\end{itemize}
	
	\vspace{0.1cm}
	
	\subsection{Nonexistence of finite-time collisions} \label{sec:3.1}
	In this subsection, we study a framework leading to the nonexistence of finite-time collisions. For this, we set real numbers ${\mathcal U},~{\mathcal L}$ and a set ${\mathcal S}$:
	\begin{align}
		\begin{aligned} \label{C-1-1-1}
			{\mathcal U} &:= \max_{i \neq j} \theta^{\infty}_{ij}+\sqrt{\frac{2N\mathcal{E}(0)}{\kappa_2}},\quad{\mathcal  L} :=\min_{i \neq j} \theta^{\infty}_{ij}-\sqrt{\frac{2N\mathcal{E}(0)}{\kappa_2}}, \\
			{\mathcal S} &:= \{ (\Theta, W) \in \bbr^{2N}:~|\theta_i - \theta_j| < {\mathcal U} < \pi  \}.  
		\end{aligned}
	\end{align}
	Note that ${\mathcal U}$ and ${\mathcal L}$ in \eqref{C-1-1-1} depend only on system parameters and initial data. First, we show that the set ${\mathcal S}$ is positively invariant along the dynamics \eqref{Ku-SB}.
	\begin{lemma}\label{L3.1}
		\emph{(Positively invariant set)}
		Suppose that initial data and system parameters satisfy
		\begin{equation} \label{C-1-2}
			(\Theta^0, W^0) \in {\mathcal S}, \quad  \kappa_0 \cos {\mathcal U} + \kappa_1 > 0, \quad \kappa_2 > 0,
		\end{equation}
		and for some $\tau\in(0,\infty]$, let $\{(\theta_i, \omega_i) \}$ be a solution to \eqref{Ku-SB} in the time-interval $[0,\tau)$. Then, one has
		\[ (\Theta(t), W(t)) \in {\mathcal S}, \quad \forall~t \in [0, \tau). \]
	\end{lemma}
	\begin{proof} The positive invariance of ${\mathcal S}$ will be verified by the continuous induction and dissipation estimate for ${\mathcal E}$ in Proposition \ref{KEE}. It follows from \eqref{C-1} that 
		\begin{align}
			\begin{aligned} \label{C-1-2-1}
				\frac{d\mathcal{E}}{dt}&= -\frac{\kappa_0}{2N}\sum_{i,j=1}^N \cos(\theta_j - \theta_i)  |\dot{\theta}_j-\dot{\theta}_i|^2-\frac{\kappa_1}{2N}\sum_{i,j=1}^N|\dot{\theta}_j-\dot{\theta}_i|^2\\
				&=-\frac{1}{2N}\sum_{i,j=1}^N \Big[ \kappa_0 \cos|\theta_j - \theta_i| + \kappa_1 \Big] |\dot{\theta}_j-\dot{\theta}_i|^2  \le0.
			\end{aligned}
		\end{align}
		Suppose that $\Theta^0$ lies in the set ${\mathcal S}$:
		\begin{equation} \label{C-1-2-2}
			|\theta^0_i - \theta^0_j| < {\mathcal U}  < \pi. 
		\end{equation}
		Now, we consider a subset ${\mathcal T}$ of $[0, \tau)$:
		\[ {\mathcal T}:=\{0< T <\tau:~~\Theta(t) \in {\mathcal S}, \quad t\in[0, T)\}. \]
		By the continuity of $\Theta = \Theta(t)$ and \eqref{C-1-2-2}, there exists $\delta > 0$ such that 
		\[ \Theta(t) \in {\mathcal S}, \quad t \in [0, \delta).  \]
		Thus, $\delta \in {\mathcal T}$, i.e., it is nonempty. Now, we define the supremum of the set ${\mathcal T}$ by $T^*$:
		\[T^*:=\sup {\mathcal T}.\]
		Then, one has for $t\in[0,T^*)$,
			\begin{equation} \label{C-1-2-3}
			|\theta_j(t) - \theta_i(t)| < {\mathcal U} < \pi \quad \mbox{and} \quad\kappa_0 \cos|\theta_j - \theta_i| + \kappa_1 \geq  \kappa_0 \cos {\mathcal U} + \kappa_1  > 0.
		\end{equation}
		Now, we claim:
		\[  T^*=\tau. \]
		{\it Proof of claim}: Suppose not, i.e., $T^*<\tau$. Then, there exist two indices $i^*,j^* \in [N]$ such that
		\begin{equation} \label{C-1-2-4}
			|\theta_{j^*}(T^*)-\theta_{i^*}(T^*)|=  {\mathcal U}.
		\end{equation}
		To derive a contradiction, since $\frac{d{\mathcal E}}{dt} \leq 0$, we consider the following two cases:
		\[
			\mbox{either}\quad\frac{d{\mathcal E}}{dt} = 0~~\mbox{for all $t \in [0, T^*)$} \quad \mbox{or}\quad\exists~T \in [0, T^*)\quad\mbox{such that} \quad \frac{d{\mathcal E}}{dt} \Big|_{t = T} < 0. 
		\]
		Below, we consider the following two cases. \newline
		
		\noindent $\bullet$~Case A: Suppose that 
		\[ \frac{d\mathcal{E}(t)}{dt} \equiv 0, \quad  \forall~t \in [0, T^*). \]
		Then, it follows from \eqref{C-1-2-1} and $\eqref{C-1-2-3}_2$ that 
		\[ \dot{\theta}_i(t)=\dot{\theta}_j(t),~\forall~i, j \in [N] \quad \mbox{and} \quad \forall~t\in[0,T^*). \]
		In this case, one obtains
		 \[ |\dot{\theta}_{j^*}(t)-\dot{\theta}_{i^*}(t)|=0,\mbox{ for all }t\in[0,T^*).\]
		 This implies
		\[|\theta_{j^*}(T^*)-\theta_{i^*}(T^*)|=|\theta_{j^*}^0 -\theta_{i^*}^0|< {\mathcal U},\]
		which gives a contradiction to \eqref{C-1-2-4}.\\
		
		\noindent $\bullet$~Case B:~Suppose there exists $T\in[0,T^*)$ such that 
		\[ \frac{d\mathcal{E}}{dt} \Big|_{t = T} <0. \]
		Since $ \frac{d\mathcal{E}}{dt} \Big|_{t = T}<0$, one has 
		\[ \mathcal{E}(T^*)<\mathcal{E}(0). \]
		On the other hand, it follows from \eqref{B-0-5} that 
		\begin{align*}
		\begin{aligned}
		\mathcal{E}_p(T^*) &=\frac{\kappa_2}{4N}\sum_{i,j =1}^{N}(|\theta_j(T^*)-\theta_i(T^*)|-\theta^{\infty}_{ij})^2  \\
		&= \frac{\kappa_2}{2N}\sum_{1\leq i < j \leq N}(|\theta_j(T^*)-\theta_i(T^*)|-\theta^{\infty}_{ij})^2 
		\leq \mathcal{E}(T^*)<\mathcal{E}(0).
		\end{aligned}
		\end{align*}
		Thus, we have
			\[ \frac{\kappa_2}{2N}(|\theta_j(T^*)-\theta_i(T^*)|-\theta^{\infty}_{ij})^2<\mathcal{E}(0), \quad i \neq j. \]
		Therefore, we have
		\[
			{\mathcal L} \leq \theta_{ij}^\infty - \sqrt{\frac{2N\mathcal{E}(0)}{\kappa_2}} \leq |\theta_j(T^*)-\theta_i(T^*)|< \theta_{ij}^\infty+\sqrt{\frac{2N\mathcal{E}(0)}{\kappa_2}}\leq {\mathcal U}.
		\]
		Again, this gives a contradiction to \eqref{C-1-2-4}, and one has $T^*=\tau$. Thus, we get the desired result.  
		 \end{proof}
	Now we are ready to provide our first main result regarding to the nonexistence of finite-time collisions. 
	\begin{theorem}\label{T3.1}
		\emph{(Nonexistence of finite-time collisions)} Suppose that initial data and system parameters satisfy
		\begin{align} \label{C-1-3}
			\begin{aligned}
				(\Theta^0, W^0) \in {\mathcal S},  \quad \mathcal{E}(0)<\frac{\kappa_2(\min_{i \ne j} \theta^{\infty}_{ij})^2}{2N}, \quad \kappa_0 \cos{\mathcal U} + \kappa_1 > 0, \quad \kappa_2 > 0,
			\end{aligned}
		\end{align}
		and for $\tau\in(0,\infty]$, let $\{ \theta_i \}_{i=1}^N$ be a smooth solution to \eqref{Ku-SB} in the time-interval $[0,\tau)$.  Then, there is no collision between each pair of all inter-particles along \eqref{Ku-SB} in the time-interval $[0, \tau]$. 
	\end{theorem}
	\begin{proof}
		Recall the potential energy:
		\begin{equation} \label{C-1-4}
			\mathcal{E}_p(t):=\frac{\kappa_2}{2N}\sum_{1 \leq i < j \leq N} \Big(|\theta_i(t)-\theta_j(t)|-\theta^{\infty}_{ij} \Big)^2. 
		\end{equation}
		Suppose there exists a collision between the pair $(i_*,j_*)$ with $i_*\ne j_*$ at time $t = \tau-$:
		\[  \theta_{i^*}(\tau-) = \theta_{j^*}(\tau-). \]
		Then, it follows from \eqref{B-0-5} that
		\begin{equation} \label{C-1-5}
				\mathcal{E}_p(\tau-) \geq \frac{\kappa_2}{2N} \Big(|\theta_{i_*}(\tau-)-\theta_{j_*}(\tau-)|-\theta^{\infty}_{i_*j_*} \Big)^2= \frac{\kappa_2 (\theta^{\infty}_{i_{*} j_{*}})^2}{2N}\geq \frac{\kappa_2(\min_{i \ne j}\theta^{\infty}_{ij})^2}{2N}.
		\end{equation}
		Now, we use Proposition \ref{KEE} and $\eqref{C-1-3}$ to find 
		\begin{equation} \label{C-1-6}
			\mathcal{E}(t)\leq \mathcal{E}(0)<\frac{\kappa_2(\min_{i \ne j} \theta^{\infty}
				_{ij})^2}{2N}, \quad t \in [0, \tau). 
		\end{equation}
		Finally, we combine \eqref{C-1-5} and \eqref{C-1-6} to find 
		\begin{align*}
			\frac{\kappa_2(\min_{i \ne j} \theta^{\infty}_{ij})^2}{2N}& \leq \mathcal{E}_p(\tau-) \leq \mathcal{E}(\tau-)\leq \mathcal{E}(0) <\frac{\kappa_2(\min_{i \ne j} \theta^{\infty}_{ij})^2}{2N},
		\end{align*} 
		which is contradictory. Hence, we obtain the desired result.
		 \end{proof}\vspace{-.4cm}
	\begin{remark} 
	\begin{enumerate}
	\item
		By combining Theorem \ref{T3.1} and continuous induction argument, one can show that under the framework \eqref{C-1-3}, finite-time collisions will not occur and then, we have a global existence of solution to \eqref{Ku-SB}. On the other hand, note that $\eqref{C-1-3}_2$ is equivalent to 
		\[   \min_{i \neq j} \theta_{ij}^{\infty}  -\sqrt{\frac{2N\mathcal{E}(0)}{\kappa_2}}  > 0. \]
		Hence, one has the strict positivity of ${\mathcal L}$ in \eqref{C-1-1-1}.
		
		\vspace{0.2cm}
		
\item 
Note that  in \eqref{C-1-2} and \eqref{C-1-3}, we required that the initial phase configuration $\{ \theta_i^0 \}$ satisfy implicit relations:
\begin{equation} \label{New-Eq-0}
|\theta^0_i - \theta^0_j| <  {\mathcal U} < \pi \quad \mbox{and} \quad \mathcal{E}(0)<\frac{\kappa_2(\min_{i \ne j} \theta^{\infty}_{ij})^2}{2N}.
\end{equation} 
Now we use the explicit representations for ${\mathcal E}(0)$ and ${\mathcal U}$:
\[
\mathcal{E}(0) =\frac{1}{2}\sum_{i=1}^{N} |\omega_i|^2+\frac{\kappa_2}{2N}\sum_{1 \leq i < j \leq N} \big(|\theta_j^0-\theta_i^0|-\theta_{ij}^\infty\big)^2, \quad {\mathcal U} =\max_{i\ne j}\theta_{ij}^\infty+\sqrt{\frac{2N\mathcal{E}(0)}{\kappa_2}}, \]
to make \eqref{New-Eq-0} to be explicit: \newline

\noindent $\diamond$~(The condition $\ref{New-Eq-0}_1$):~We use $\displaystyle {\mathcal E}(0) \geq \frac{\kappa_2}{2N}\sum_{1 \leq i < j \leq N} \big(|\theta_j^0-\theta_i^0|-\theta_{ij}^\infty\big)^2$ to find
\[
0<\pi-{\mathcal U} = \pi - \Big[   \max_{i\ne j}\theta_{ij}^\infty+\sqrt{\frac{2N\mathcal{E}(0)}{\kappa_2}}  \Big] \leq \pi-\max_{i\ne j}\theta_{ij}^\infty-\sqrt{ \sum_{1 \leq i < j \leq N} \big(|\theta_j^0-\theta_i^0 |-\theta_{ij}^\infty\big)^2}.
\]
This implies 
\begin{equation} \label{New-Eq-0-1}
\big(\pi-\max_{i \neq j}\theta_{ij}^\infty\big)^2\ge\sum_{1 \leq i <j \leq N} \Big(|\theta_j^0-\theta_i^0 |-\theta_{ij}^\infty \Big)^2.
\end{equation}

\noindent (The condition $\ref{New-Eq-0}_2$):~Note that 
\[ 0<\frac{\kappa_2\min_{i,j}(\theta_{ij}^\infty)^2}{2N}-\mathcal{E}(0)<\frac{\kappa_2}{2N}\Big(\min_{i\ne j}(\theta_{ij}^\infty)^2- \sum_{1 \leq i <j \leq N}\big(|\theta_j^0-\theta_i^0|-\theta_{ij}^\infty\big)^2\Big).
\]
This implies 
\begin{equation} \label{New-Eq-0-2}
\min_{i\ne j}(\theta_{ij}^\infty)^2\ge\sum_{1 \leq i < j \leq N}\big(|\theta_j^0-\theta_i^0|-\theta_{ij}^\infty\big)^2.
\end{equation}
Finally, we combine \eqref{New-Eq-0-1} and \eqref{New-Eq-0-2} to obtain 
	\begin{equation}\label{New-Eq-0-3}
	\sum_{1 \leq i <j \leq N}^N\big(|\theta_j^0 -\theta_i^0|-\theta_{ij}^\infty\big)^2 \leq	\min \Big \{ \min_{\i\ne j}(\theta_{ij}^\infty)^2,~(\pi-\max_{i \neq j}\theta_{ij}^\infty)^2\Big \}.
	\end{equation}
\end{enumerate}		
	\end{remark}
As a direct corollary of Theorem \ref{T3.1}, we have a uniform positive lower bound for $|\theta_i - \theta_j|$.
	\begin{corollary}\label{C3.1}
		Suppose that initial data and coupling strengths satisfy \eqref{C-1-3}, and let $\{ \theta_i \}$ be a global smooth solution to \eqref{Ku-SB}. Then, one has 
		\[  \inf_{0 \leq t < \infty} \min_{i \neq j} |\theta_j(t) -\theta_i(t)| \geq {\mathcal  L}
		> 0. \]
	\end{corollary}
	\begin{proof}  Since $|\theta_j-\theta_i|-\theta^{\infty}_{ij}$ is invariant under the index exchange transformation $(i, j)\leftrightarrow(j, i)$, the potential energy satisfies
			\[  \mathcal{E}_p =\frac{\kappa_2}{2N} \sum_{1 \leq i < j \leq N} \Big(|\theta_j-\theta_i|- \theta^{\infty}_{ij} \Big)^2. \] 
		Now, we use the above relation and Proposition \ref{KEE} to have the following relation: 
		\[  {\mathcal E}(0) \geq {\mathcal E}(t) \geq {\mathcal E}_p(t) \geq  \frac{\kappa_2}{2N} \Big(|\theta_j(t)-\theta_i(t)|- \theta^{\infty}_{ij} \Big)^2,\]
		for $i\ne j$. This and the definition of ${\mathcal L}$ yield 
		\[
			{\mathcal L} = \min_{k \neq l} \theta^{\infty}_{kl}  -\sqrt{\frac{2N {\mathcal E}(0)}{\kappa_2}}  \leq  \theta^{\infty}_{ij}  -\sqrt{\frac{2N {\mathcal E}(0)}{\kappa_2}} \leq   |\theta_j(t) -\theta_i(t)|,\quad\mbox{for}\quad t\in(0,\infty). 
		\]
	This implies the desired estimate.
		 \end{proof}
	\subsection{Complete synchronization} \label{sec:3.2}
	In this subsection, we present complete synchronization to \eqref{Ku-SB} by using Babalat's lemma and energy estimate in Proposition \ref{KEE}. First, we recall various concepts for synchronization.
	\begin{definition}  \label{D3.1}
		Let $\Theta = \{ \theta_i \}$ be a time-dependent phase ensemble. 
		\begin{enumerate}
			\item
			$\Theta = \Theta(t)$ is a phase-locked state if all relative phase differences are constant:
			\[ | \theta_i(t) - \theta_j(t) |= |\theta_i^{0} - \theta_j^{0}|,  \quad\forall~t \geq 0, ~~\forall~i, j \in [N]. \] 
			\item
			$\Theta = \Theta(t)$ exhibits (asymptotic) complete phase-locking, if the relative phase differences converge as $t \to \infty$:
			\[ \exists \lim_{t \to \infty} (\theta_i(t) - \theta_j(t)), \quad \forall~i, j \in [N]. \]
			\item
			$\Theta = \Theta(t)$ exhibits complete synchronization, if the relative frequency differences converge to zero as $t \to \infty$:
			\[  \lim_{t \to \infty} \max_{i,j}  |{\dot \theta}_i(t) - {\dot \theta}_j(t)| = 0. \]
	\end{enumerate}
	\end{definition}
Now, we recall the Babalat lemma in \cite{B} to be used in Theorem \ref{T3.2} and Theorem \ref{T4.2}.
	\begin{lemma}\emph{\cite{B}}\label{Barbalat}~\emph{(Barbalat's lemma)}
		Let $f:(0,\infty)\to\mathbb{R}$ be a uniformly continuous function such that 
		\[ \int_0^{\infty} f(t) dt < \infty. \]
		Then, one has
		\[ \lim_{t \to \infty} f(t) = 0. \]
	\end{lemma}
	\begin{lemma}\label{L3.3}
		Suppose that initial data and system parameters satisfy \eqref{C-1-3}, and let $\{ \theta_i \}$ be a global smooth solution to \eqref{Ku-SB}. Then, for $i \neq j$,  one has
		\[  \Big \| \frac{d}{dt}\left|\omega_j(t)-\omega_i(t)\right|^2 \Big \|_{L^{\infty}(\bbr_+)} <\infty. \]
		Thus, the map $t\mapsto\sum_{i,j}|w_j(t)-w_i(t)|^2$ is uniformly continuous.
	\end{lemma}
	\begin{proof} 
		Note that $\omega_i$ and $\omega_j$ satisfy 
		\begin{align*}
			\begin{aligned}
				{\dot \omega}_i  &=   \frac{1}{N} \sum_{k=1}^{N} \Big [ \kappa_0 \cos(\theta_k - \theta_i) + \kappa_1 \Big] (\omega_k - \omega_i)+ \frac{\kappa_2}{N}\sum_{k=1}^{N} \Big[ |\theta_k - \theta_i|- \theta^{\infty}_{ik} \Big] \textup{sgn}(\theta_k - \theta_i),  \\
				{\dot \omega}_j &=   \frac{1}{N} \sum_{k=1}^{N} \Big [ \kappa_0 \cos(\theta_k - \theta_j) + \kappa_1 \Big] (\omega_k - \omega_j) + \frac{\kappa_2}{N}\sum_{k=1}^{N} \Big[ |\theta_k - \theta_j|- \theta^{\infty}_{jk} \Big] \textup{sgn}(\theta_k - \theta_j).
			\end{aligned}
		\end{align*}		
		These and the Cauchy-Schwarz inequality imply
		\begin{align}
		\begin{aligned} \label{C-1-9}
				&\Big| \frac{d}{dt}\left|\omega_j(t)-\omega_i(t)\right|^2 \Big| =\Big|(\dot{\omega}_j-\dot{\omega}_i) \cdot(\omega_j-\omega_i)\Big|\\
				&\hspace{1cm} \leq \frac{(\kappa_0+\kappa_1)}{N}\sum_{k=1}^N \Big(|\omega_k-\omega_j| +|\omega_k-\omega_i| \Big) \cdot |\omega_j-\omega_i|    \\
				&\hspace{1.6cm}+\frac{\kappa_2|\omega_j-\omega_i|}{N}\left(\sum_{k=1}^N\Big||\theta_k-\theta_j|-\theta_{jk}^\infty\Big|+\sum_{k=1}^N\Big||\theta_k-\theta_i|-\theta_{ik}^\infty\Big|\right).
			\end{aligned}
		\end{align}
		Due to Lemma \ref{L3.1}, we have
		\[ \sup_{0 \leq t < \infty} |\theta_i(t) - \theta_j(t)| \leq {\mathcal U} < \pi. \]
		This yields
		\[ \kappa_0 \cos(\theta_j-\theta_i) + \kappa_1 \geq \kappa_0 \cos {\mathcal U} + \kappa_1 > 0. \]
		Since the energy production rate $\mathcal{P}$ is nonnegative, Proposition \ref{KEE} leads to
		\[\mathcal{E}(t)=\mathcal{E}_k(t)+\mathcal{E}_p(t)\le\mathcal{E}(0),\quad\forall ~t\ge0.\]
		With the nonnegativity of  $\mathcal{E}_k$ and $\mathcal{E}_p$, the uniform boundedness of them can be obtained:
		\begin{equation} \label{C-1-10} 
		\sup_{0\le t\le\infty}\frac{1}{2}\sum_{i=1}^N|w_i(t)|^2<\infty\quad\mbox{and}\quad\sup_{0\le t\le\infty}\frac{\kappa_2}{4N}\sum_{i,j=1}^N\Big(|\theta_j(t)-\theta_i(t)|-\theta_{ij}^\infty\Big)^2<\infty.
		\end{equation}
		Finally, we combine \eqref{C-1-9} and \eqref{C-1-10} to get the desired estimate.
		 \end{proof}
	Now, we present the proof of complete synchronization of \eqref{Ku-SB}.
	\begin{theorem}\label{T3.2}~\emph{(Complete synchronization)}
		Suppose that initial data and coupling strengths satisfy \eqref{C-1-3} together with, and let $\{ \theta_i \}$ be a global smooth solution to \eqref{Ku-SB}. Then, the following assertions hold.
		\begin{enumerate}
		\item
			Complete synchronization emerges asymptotically:
			\[
			\lim_{t\to\infty} \max_{i,j} |{\dot \theta}_j(t)- {\dot \theta}_i(t)|= 0.
			\]
			\item	
			If $\displaystyle \sum_{i=1}^{N} \omega_i^0 = 0$, then we have
			 \[ \lim_{t \to \infty} \max_{1 \leq i \leq N} |\omega_i(t)| = 0. \]
	\end{enumerate}
	\end{theorem}
	\begin{proof}
		\noindent (i)~It follows from Proposition \ref{KEE} that 
		\[
		\int_{0}^{\infty} \sum_{i,j =1 }^{N}  |\omega_j(t) -\omega_i(t)|^2 dt \leq \frac{2N \mathcal{E}(0)}{\kappa_0 \cos {\mathcal U} + \kappa_1} < \infty,
		\]
		for $i, j \in [N]$.
		We apply Lemma \ref{L3.3} to show that the map $t \mapsto  \sum_{i,j} |\omega_j(t) - \omega_i(t)|^2$ is uniformly continuous. Thus, we can employ Lemma \ref{Barbalat} to prove the desired estimate:
		\begin{equation} \label{C-1-7}
			\lim_{t \to \infty} \sum_{i,j=1}^{N}  |\omega_j(t) - \omega_i(t)|^2 = 0.
		\end{equation}
		\noindent (ii)~We use Lemma \ref{L2.1} to get 
		\[  \sum_{i=1}^{N} \omega_i(t) =  \sum_{i=1}^{N} \omega_i^0=0, \quad t \geq 0. \]
		This implies 
		\begin{equation} \label{C-1-8}
			\sum_{i,j=1}^N |\omega_i - \omega_j|^2 =  2 N \sum_{i=1}^{N} |\omega_i|^2 - 2\Big| \sum_{i=1}^{N} \omega_i \Big|^2 =  2 N \sum_{i=1}^{N} |\omega_i|^2.
		\end{equation}
		Consequently, we combine \eqref{C-1-7} and \eqref{C-1-8} to obtain the desired estimate. 
		 \end{proof}
		 As a directly corollary of Theorem \ref{T3.2}, we can infer asymptotic behaviors of kinetic, potential and total energies. 
		 \begin{corollary} \label{C3.2}
		Suppose that initial data and coupling strengths satisfy \eqref{C-1-3} and $\displaystyle \sum_{i=1}^{N} \omega_i^0 = 0$, and  let $\{ (\theta_i, \omega_i) \}$ be a global smooth solution to \eqref{Ku-SB}. Then, there exists an nonnegative constant ${\mathcal E}^{\infty}$ such that 
		\[ \lim_{t \to \infty} {\mathcal E}_k(t) = 0, \quad  \lim_{t \to \infty} {\mathcal E}_p(t) = {\mathcal E}^{\infty}, \quad  \lim_{t \to \infty} {\mathcal E}(t) = {\mathcal E}^{\infty}.
		\]
		  \end{corollary}
		 \begin{proof} (i) It follows from the second assertion in Theorem \ref{T3.2} that 
		   \[ \lim_{t \to \infty} |\omega_i(t)| = 0, \quad \mbox{for all $i \in [N]$}. \]
		Thus, the total kinetic energy tends to zero asymptotically:
		\begin{equation} \label{New-Eq-1}
		 \lim_{t \to \infty} {\mathcal E}_k(t) = \frac{1}{2} \sum_{i=1}^{N}  \lim_{t \to \infty} |\omega_i(t)|^2 = 0.
		 \end{equation}
		(ii) It follows from \eqref{C-1-3} and Proposition \ref{KEE} that a nonnegative energy ${\mathcal E}$ is nonincreasing in time $t$. Thus, there exists ${\mathcal E}^{\infty} \in [0, {\mathcal E}(0))$ such that 
		\begin{equation} \label{New-Eq-2}
		 \lim_{t \to \infty} {\mathcal E}(t) = {\mathcal E}^{\infty}.
		 \end{equation}
		Finally, we combine \eqref{New-Eq-1} and \eqref{New-Eq-2}, and use ${\mathcal E} = {\mathcal E}_k + {\mathcal E}_p$ to see
		\[  \lim_{t \to \infty} {\mathcal E}_p(t) = {\mathcal E}^{\infty}. \]
		 
		 \end{proof}
%
%
        \section{The Cucker-Smale model with a bonding force} \label{sec:4}
	\setcounter{equation}{0}
	 In this section, we study the nonexistence of finite-time collisions and emergent dynamics of the Cucker-Smale model with a bonding control. Thanks to Lemma \ref{L2.2}, without loss of generality, we may assume
	\[ \sum_{i=1}^N \bv_i(t) = 0,  \quad  t \geq 0. \]
	\subsection{Nonexistence of finite-time collisions} Recall that 
	\[ \br_{ij} := \bx_i-\bx_j, \quad  r_{ij} :=\| \br_{ij}\|, \quad i, j \in [N], \]
	and consider velocity dynamics in \eqref{CSB}:
	\begin{align} \label{D-0}
		\begin{aligned}
			\displaystyle {\dot{\bv}_i} =&\frac{\kappa_0}{N}\sum_{j=1}^{N}\psi(r_{ij})\left(\bv_j- \bv_i\right) + \frac{\kappa_1}{N}\sum_{j=1}^{N}\frac{\langle \bv_j-\bv_i,\br_{ji}\rangle}{r^2_{ij}}\br_{ji} +\frac{\kappa_2}{N}\sum_{j=1}^{N}\frac{r_{ij}-d^{\infty}_{ij}}{r_{ij}}\br_{ji}.
		\end{aligned}
	\end{align}
	Note that $r_{ij}$ appears in the denominator in the R.H.S. of \eqref{D-0} and the continuity of $\frac{\br_{ji}}{r_{ji}}$ breaks down when $r_{ji}=0$, i.e. at the instant when two particles collide. Therefore, we have to make sure that $r_{ij} > 0$ in any finite-time interval to get the global existence of classical solutions in the framework of the Cauchy-Lipschitz theory. Similar to \eqref{C-1-1-1}, we set 
	\begin{align}
		\begin{aligned} \label{D-0-0}
			U &:= \max_{i \neq j}d^{\infty}_{ij}+\sqrt{\frac{2N E(0)}{\kappa_2}}, \quad L :=\min_{i \neq j} d^{\infty}_{ij}-\sqrt{\frac{2N E(0)}{\kappa_2}}, \\
			S &:= \Big \{ (X, V)  \in \bbr^{2dN}:~\max_{i \neq j} \| \bx_i - \bx_j \| \leq U \Big \}.  
		\end{aligned}
	\end{align}
	Note that $L$ and $U$ are slightly different from ${\mathcal L}$ and ${\mathcal U}$ in \eqref{C-1-1-1}. 
	\begin{theorem}\label{T4.1}
		\emph{(Nonexistence of finite-time collisions)}
		Suppose initial data and system parameters satisfy
		\begin{align} \label{D-3-1}
			\begin{aligned}
				\min_{i \neq j} \|\bx_i^0  - \bx_j^0 \| > 0, \quad  \min_{i \neq j} d^{\infty}_{ij} > \sqrt{\frac{2N E(0)}{\kappa_2}}, \\
				(X^0, V^0) \in S,\quad\kappa_0 > 0, \quad  \kappa_1 > 0, \quad \kappa_2 > 0,
			\end{aligned}
		\end{align}
		and for $\tau\in(0,\infty]$, let $\{(\bx_i,\bv_i)\}_{i=1}^N$ be a smooth solution to \eqref{CSB} -\eqref{A-2} in the time-interval $[0,\tau)$. Then we have
		\begin{equation}\label{D-4}
			0 < L \leq \| \bx_i(t) - \bx_j(t) \| 
			\leq  U, \quad \forall~ t \in [0, \tau), 
		\end{equation}
		for $i,j\in[N]$. In particular, this implies the nonexistence of finite-time collisions. 
	\end{theorem}
	\begin{proof}
		We use Proposition \ref{CSEE} to bound the potential energy by
		\begin{align} \label{D-5}
			\begin{aligned}
				E_p(t)&\leq E_k(t)+\frac{\kappa_2}{2N}\sum_{1 \leq i < j \leq N}( \| \bx_i(t) - \bx_j(t) \| -d^{\infty}_{ij})^2\leq E(0).
			\end{aligned}
		\end{align}
		For $i, j \in [N]$, we use \eqref{D-5} to find 
		\begin{align*}\label{L4.2.3}
			\frac{\kappa_2}{2N}(  &\| \bx_i(t) - \bx_j(t) \|  -d^{\infty}_{ij})^2\leq \frac{\kappa_2}{2N}\sum_{1 \leq i < j \leq N} \Big ( \| \bx_i(t) - \bx_j(t) \|  -d^{\infty}_{ij} \Big )^2 \leq E(0).
		\end{align*}
		This results in
		\[  \Big|  \| \bx_i(t) - \bx_j(t) \|  -d^{\infty}_{ij} \Big|  \leq  \sqrt{\frac{2N E(0)}{\kappa_2}}, \]
		or equivalently, 
		\begin{align*}
		d^{\infty}_{ij}-\sqrt{\frac{2N E(0)}{\kappa_2}}\leq \|x_i(t)-x_j(t)\|\leq d^{\infty}_{ij}+\sqrt{\frac{2N E(0)}{\kappa_2}}, \quad i, j \in [N]. 
		\end{align*}
		This leads to
		\begin{equation} \label{D-5-1}
			L \leq  \| \bx_i(t) - \bx_j(t) \|  \leq U, ~~ t \in [0, \tau), ~~ i, j \in [N].
		\end{equation}
		On the other hand, the condition $\eqref{D-3-1}_2$ implies
		\begin{equation} \label{D-5-2}
			L > 0.
		\end{equation}
		Finally, we combine \eqref{D-5-1} and \eqref{D-5-2} to derive \eqref{D-4}. 
		 \end{proof}
	\begin{remark}\label{R4.1}
		Although the overall approach in the current subsection is similar to that of \cite{P-K-H}, it turns out that the control of individual $d^{\infty}_{ij}$ leads to the stability of the system, represented by Theorem \ref{T4.1}. We recall that energy is minimized if $\|\bx_i-\bx_j\| \approx d^{\infty}_{ij}$ for each $i,j$, but this is fundamentally impossible under the constraint of identical parameter $d^{\infty}_{ij} \equiv 2R~(i \neq j)$ unless the number of particle is small. This is because there is no feasible configuration to place $N$ particles in a way that every $N(N-1)/2$ distances are identical, in physical space of small dimension. Hence, if $d^{\infty}_{ij} \equiv 2R$ as in \cite{P-K-H}, then although the system will lead to state corresponding to the local minimizer of the energy functional, the potential energy $\mathcal{E}_p$ is forced to have certain lower bound, reflecting the infeasibility of the configuration satisfying $d^{\infty}_{ij} \equiv 2R$. To summarize, the stability and the well-posedness of \eqref{CSB} is influenced by the feasibility of the configuration parameter $\{ d^{\infty}_{ij}\}$, which is realized by Theorem \ref{T4.1}.
	\end{remark}
	
	\subsection{Asymptotic flocking} \label{sec:4.2}
	In this subsection, we present the asymptotic flocking estimate of \eqref{CSB}. For this, we first recall the concept of asymptotic flocking as follows.
	\begin{definition}\label{D4.1}
		Let $Z=\{(\bx_i,\bv_i)\}$ be a global smooth solution to \eqref{CSB}-\eqref{A-2}. The configuration $Z$ exhibits (asymptotic) flocking if the following two conditions hold.
		\begin{enumerate}
			\item
			(Spatial cohesion): the relative distances are uniformly bounded in time: 
			\[ \sup_{0 \leq t < \infty} \max_{i,j}  \|\bx_i(t) - \bx_j(t) \| <\infty. \]
			\item
			(Asymptotic velocity alignment): the relative velocities tend to zero asymptotically: 
			\[
			\lim_{t \to \infty} \max_{i,j} \|\bv_j(t)-\bv_i(t)\|=0.
			\]	
		\end{enumerate}
	\end{definition}
%
%
	First, we begin with the following elementary lemma.
	\begin{lemma} \label{L4.1}
		Let $\{(\bx_i,\bv_i)\}$ be a global smooth solution to \eqref{CSB} - \eqref{A-2}. Then, one has 		
		\[ \left \|\frac{d}{dt}\|\bv_i(t) -\bv_j(t) \|^2\right \|_{L^{\infty}(\bbr_+)} <\infty, \quad \mbox{for}~ i, j \in [N]. \] 
		In particular,  the map $t \mapsto \|\bv_i-\bv_j\|^2$ is uniformly continuous.
	\end{lemma}
	\begin{proof}
	We use the same arguments in Lemma \ref{L3.3}. Note that 
		\begin{align}\label{D-5-3}
		\begin{aligned}
			&\Big|(\dot{{\bv}}_j-\dot{{\bv}}_i)\cdot({\bv}_j-{\bv}_i)\Big| \\
			&\le\Bigg[\frac{\kappa_0\psi_M}{N}\sum_{k=1}^N\Big(\|{\bv}_k-{\bv}_j\|+\|{\bv}_k-{\bv}_i\|\Big)+\frac{\kappa_1}{N}\sum_{k=1}^N\Big(\|{\bv}_k-{\bv}_j\|+\|{\bv}_k-{\bv}_i\|\Big)\\
			&\hspace{5.9cm}+\frac{\kappa_2}{N}\sum_{k=1}^N\Big(\Big|r_{kj}-d_{kj}^\infty\Big|+\Big|r_{ki}-d_{ki}^\infty\Big|\Big)\Bigg]\|{\bv}_j-{\bv}_i\|\\
			& \le \Bigg[\frac{\kappa_0\psi_M}{N} \sum_{k=1}^N\Big(\|{\bv}_k-{\bv}_j\|+\|{\bv}_k-{\bv}_i\|\Big)+\frac{\kappa_1}{N}\sum_{k=1}^N\Big(\|{\bv}_k-{\bv}_j\|+\|{\bv}_k-{\bv}_i\|\Big)\Bigg]\|{\bv}_j-{\bv}_i\|\\
			&\hspace{0.5cm}+\frac{\kappa_2}{N}\left(\left(\sum_{k=1}^N\Big|r_{kj}-d_{kj}^\infty\Big|^2\right)^{1/2}+\left(\sum_{k=1}^N\Big|r_{ki}-d_{ki}^\infty\Big|^2\right)^{1/2}\right)\left(\sum_{k=1}^N\|{\bv}_j-{\bv}_i\|^2\right)^{1/2},
		\end{aligned}
		\end{align}
		where $\psi_M$ is a finite upper bound of $\psi$. Due to Proposition \ref{CSEE}, we can get uniform boundedness of $E_k$ and $E_p$:
		\begin{align} \label{D-5-4}
			\begin{aligned}
				&\sup_{0\le t<\infty}\sum_{k=1}^N\|{\bv}_k\|^2<\infty\quad\mbox{and}\quad\sup_{0\le t<\infty}\Big(\|x_j-x_i\|-d_{ij}^\infty\Big)^2<\infty.
			\end{aligned}
		\end{align}
		Consequently, we combine \eqref{D-5-3} with \eqref{D-5-4} to obtain the desired estimate. 
		 \end{proof}
	\hspace{.1cm} Now, we present our fourth main result on the emergence of asymptotic flocking. 
	\begin{theorem}\label{T4.2} 
		\emph{(Global flocking)}
		Suppose that initial data, system parameters and communication weight function satisfy \eqref{D-3-1}, and let $\{(\bx_i,\bv_i)\}$ be a global smooth solution to \eqref{CSB}. Then, the following assertions hold.
		\begin{enumerate}
			\item
			Asymptotic flocking occurs:
			\begin{align*}
				&\sup_{0 \leq t < \infty} \max_{i,j} \|\bx_i(t) - \bx_j(t) \| <\infty,\quad\lim_{t \to \infty} \max_{i,j} \|\bv_j(t)-\bv_i(t)\|=0. 
			\end{align*}
			\item
			If the initial total momentum is zero, that is, $\sum_{i=1}^N\bv_i^0=0$, then we have
			\[ \lim_{t\to\infty} \max_{1 \leq i \leq N} \|\bv_i(t) \|=0. \]
		\end{enumerate}
	\end{theorem}
	\begin{proof}
		\noindent (i) Since initial data and system parameters satisfy the same conditions \eqref{D-3-1} as in Theorem \ref{T4.1}, the spatial cohesion can be followed from \eqref{D-4}:
		\begin{equation} \label{D-6}
			\sup_{0 \leq t < \infty} \max_{i,j} \| \bx_i(t) - \bx_j(t) \| \leq U.
		\end{equation}
		For the velocity alignment, we use Proposition \ref{CSEE} to obtain 
		\begin{align*}
		\frac{dE}{dt}&\leq -\frac{\kappa_0}{2N}\sum_{i,j =1}^{N}\psi(r_{ji})\| \bv_j- \bv_i\|^2\leq -\frac{\kappa_0 \psi_m}{2N}\sum_{i,j =1}^{N}\|\bv_j- \bv_i\|^2. 
		\end{align*}
		Then, one can show that
		\[
		\int_{0}^{\infty}\sum_{i=1}^{N}\|\bv_j(t) -\bv_i(t) \|^2dt \leq \frac{2N E(0)}{ \kappa_0 \psi_m}< \infty.
		\]
		Thanks to Lemma \ref{L4.1}, the mapping $t\mapsto\sum_{i=1}^{N}\|\bv_j(t) -\bv_i(t) \|^2$ is uniformly continuous. Thus, we use Barbalat's Lemma to get
		
		\begin{equation} \label{D-7}
			\lim_{t \to \infty} \sum_{i,j=1}^{N} \|\bv_i(t) -\bv_j(t) \|^2 = 0.
		\end{equation}
		This yields the desired velocity alignment. Accordingly, we combine \eqref{D-6} and \eqref{D-7} to get the desired flocking estimates. \newline
		
		\noindent (ii)~Suppose initial total momentum $\sum_{i=1}^{N} \bv_i^0$ is zero. Then, by Lemma \ref{L2.2}, one has 
		\[ \sum_{i=1}^{N} \bv_i(t) = 0, \quad \forall~t >0. \]
		This leads to
		\begin{align*}
			\sum_{i,j=1}^{N} \|\bv_i(t) -\bv_j(t) \|^2 &= 2N  \sum_{i=1}^{N} \| \bv_i(t) \|^2 - 2 \Big|  \sum_{i=1}^{N} \bv_i(t) \Big|^2=  2N  \sum_{i=1}^{N} \| \bv_i(t) \|^2.
		\end{align*}
		Letting $t \to \infty$, we use \eqref{D-7} and the above relation to find
		\[ \lim_{t \to \infty}  \sum_{i=1}^{N} \| \bv_i(t) \|^2 = 0.\]
		 \end{proof}
	\begin{remark} By the conservation of momentum and velocity alignment estimate, one can see that the velocities tend to the average initial momentum asymptotically:
	\[ \bv_i(t) \to \frac{1}{N} \sum_{i=1}^{N} \bv_i^0, \quad \mbox{as $t \to \infty$ for all $i \in [N]$}.	
	\]

	\end{remark}

	\subsection{Convergence property of two-particle system} \label{sec:4.3}
	In this subsection, we study the convergence of relative distances for the Cauchy problem to the two-particle system on the real line:
	 \begin{equation}
		\begin{cases} \label{New-D-1}
			\displaystyle \dot{x}_1 = v_1, \quad  \dot{x}_2 = v_2, \quad t>0,\\
			\displaystyle \dot{v}_1 =\frac{\kappa_0}{2} \psi(|x_2- x_1|)\left(v_2- v_1 \right) + \frac{\kappa_1}{2} 
			(v_2-v_1) \Big( \frac{(x_2-x_1)}{|x_2-x_1|} \Big)^2 \\
			\displaystyle \hspace{0.7cm} + \frac{\kappa_2}{2} (|x_2-x_1|- d^{\infty}_{12}) \frac{(x_2-x_1)}{|x_2-x_1|}, \\
			\displaystyle \dot{v}_2 =\frac{\kappa_0}{2} \psi(|x_1- x_2|)\left(v_1- v_2 \right) + \frac{\kappa_1}{2} 
			(v_1-v_2) \Big( \frac{(x_1-x_2)}{|x_1-x_2|} \Big)^2 \\
			\displaystyle \hspace{0.7cm} + \frac{\kappa_2}{2} (|x_1-x_2|- d^{\infty}_{12}) \frac{(x_1-x_2)}{|x_1-x_2|}, \\
			\displaystyle (x_1, v_1)(0)= (x_1^0, v_1^0), \quad (x_2, v_2)(0)= (x_2^0, v_2^0).
	\end{cases}
	\end{equation}
	So far, we focused on the collision avoidance so that we can use the Cauchy-Lipchitz theory to guarantee well-posedness under some sufficient framework in terms of system parameters and initial data. However, collisions between particles can occur some situations. In this situation, the R.H.S. of \eqref{New-D-1} becomes discontinuous at the position $x_1 = x_2$. Thus, we reinterpret \eqref{New-D-1} as the system of \emph{differential inclusion} and using the Filippov theory, we can still construct absolutely continuous solution to \eqref{New-D-1}. In what follows, we will show that this Filippov solution exhibits the convergence property of relative distance $|x_2 - x_1|$ as $t \to \infty$.  First, we set
	\[
		\psi \equiv 1, \quad x(t):=x_1(t)-x_2(t), \quad v(t):=v_1(t)-v_2(t), \quad  d^{\infty} := d^{\infty}_{12}.
	\]
	Then, it follows from \eqref{New-D-1} that $(x, v)$ satisfies 
	\begin{equation}
		\begin{cases} \label{CSB1}
			\displaystyle \dot{x}  =v,\quad t>0,\\
			\displaystyle \dot{v} = -{\kappa_0}\left(\frac{x}{|x|}\right)^2v
			-\kappa_1 \left(\frac{x}{|x|}\right)^2 v
			-\kappa_2 (|x| - d^{\infty}) \frac{x}{|x|}, \\
			\displaystyle (x, v)(0) = (x^0, v^0).
		\end{cases}
	\end{equation}
	Equivalently, \eqref{CSB1} can be rewritten as 
	\begin{equation} \label{CSB2}
	\ddot{x}
			= \begin{cases}
			-(\kappa_0 + \kappa_1) \dot{x} -\kappa_2 x + \kappa_2 d^{\infty}  \quad  & \text{if} ~ x > 0,\\
			-(\kappa_0 + \kappa_1) \dot{x} - \kappa_2 x -\kappa_2 d^{\infty} \quad  & \text{if} ~ x < 0.
			\end{cases}
	\end{equation}
	Once we have $(x, v)$, we use the balance laws:
	\[ v_1(t) + v_2(t) = v_1^0 + v_2^0, \quad x_1(t) + x_2(t) = x_1^0 + x_2^0 + (v_1^0 + v_2^0) t, \quad t \geq 0, \]
	then we can determine $(x_1, v_1)$ and $(x_2, v_2)$.  Since the set $\{ (x,v) \mid x =0 \}$ is of measure zero in $(x,v)$-space, Definition \ref{Fdef} allow us to construct an absolutely continuous solutoin to \eqref{CSB1}. 
	
\subsubsection{A short summary of Filippov's theory} \label{sec:4.3.1} In this part, we present a brief summary of Filippov's generalization solutions to a differential equation on Euclidean space $\bbr^n$ with discontinuous right-hand side. First, we begin with several basic definitions. 
\begin{definition}[Filippov solution \cite{Fi-0}]\label{Fdef}
Consider a system of differential equations for $z = z(t) \in \bbr^{n}$:
\begin{equation} \label{New-D-2}
\dot{z}(t) = Z(z(t)),  \quad t >0,
\end{equation}
where a vector field $Z:\bbr^n \to \bbr^n$ is possibly \emph{discontinuous}. \newline
\begin{enumerate}
		\item Let $\mathcal{P}(\bbr^n)$ be the power set of $\bbr^n$. Then the \emph{Filippov set-valued map} $F[Z]:\bbr^n \to \mathcal{P}(\bbr^n)$ is defined by
		\[
			F[Z](z) := \bigcap_{\delta>0}\bigcap_{|S|=0} \overline{\mathrm{co}}\{ Z(B_\delta(z) - S) \},
		\]
		where $\overline{\mathrm{co}}$ is a closed convex hull and $| \, \cdot \, |$ is the Lebesgue measure on $\mathbb{R}^n$. \newline
		\item $z = z(t)$ is a \emph{Filippov solution} to the Cauchy problem \eqref{New-D-2} on $[0,T] \in \bbr$, if $z$ is absolutely continuous and satisfies the following {\it differential inclusion} for almost every $t \in [0,T]$:
		\[
			\dot{z}(t) \in F[Z](z(t)).
		\]
	\end{enumerate}
\end{definition}
	Roughly speaking, since $x$ itself can be excised while constructing a convex hull, the Filippov solution can be understood as a generalized solution that refers to the vector field at neighborhood only. Therefore, we may expect that behavior at a discontinuous point can be treated from its neighborhood as far as nearby behavior is consistent, which we clarify in the following proposition. 
	\begin{definition}
	A vector field $Z:\bbr^n \to \bbr^n$ is \emph{piecewise continuous}, if there exists a finite collection of disjoint, open, and connected sets $\mathcal{D}_1,\cdots,\mathcal{D}_m \subset \bbr^n$ such that
		\begin{enumerate}
			\item $\displaystyle \bigcup_{k=1}^m\overline{\mathcal{D}}_k=\bbr^n$, and $Z$ is continuous on each $\mathcal{D}_k$.
			\item restriction of $Z$ to each $\mathcal{D}_k$ admits a continuous extension to the closure $\overline{\mathcal{D}}_k$, which is denoted by $Z \Big|_{\overline{\mathcal{D}}_k}$.
		\end{enumerate}
	\end{definition}
	\begin{proposition}\label{P4.1}\cite{Co,Fi-0}
		Let $Z:\bbr^n \to \bbr^n$ be a piecewise continuous vector field covered by $\mathcal{D}_1$ and $\mathcal{D}_2$. Let $S_Z := \partial\mathcal{D}_1 = \partial\mathcal{D}_2$ be the set of points at which $Z$ is discontinuous. Suppose that
		\begin{enumerate}
			\item $S_Z$ is a ${\mathcal C}^2$-manifold.
			\item For $i=1,2$, $Z \Big|_{\overline{\mathcal{D}}_i}$ is continuously differentiable on $\mathcal{D}_i$ and $Z \Big|_{\overline{\mathcal{D}}_1}- Z \Big|_{\overline{\mathcal{D}}_2}$ is continuously differentiable on $S_Z$.
			\item For each $z \in S_Z$, either $Z \Big|_{\overline{\mathcal{D}}_1}$ points into $\mathcal{D}_2$ or $Z \Big|_{\overline{\mathcal{D}}_2}$ points into $\mathcal{D}_1$.
		\end{enumerate}
	Then, $\dot{z}(t)=Z(z(t))$ has a unique Filippov solution starting from each initial data.
	\end{proposition}
	
\subsubsection{A Filippov solution with convergence property} \label{sec:4.3.2}	 In this part, we construct a unique Filippov solution to \eqref{CSB1} with the following convergence property:
\[ \lim_{t \to \infty} |x_1(t) - x_2(t)| = d^{\infty}. \]
Now we return to \eqref{CSB1}. From Definition \ref{Fdef}, it is straightforward to construct a Filippov solution $(x,v)$. 

Below, we sketch the construction procedure of a Filippov solution as follows. First, we define a local classical solution, which is well defined until the sign of $x$ changes. Let $v$ be a velocity when sign of $x$ alters. If $v \neq 0$, then we shift the dynamic and define a new local solution with initial datum $(0,v)$. Then, by repeating this procedure and concatenating the solutions so far, we obtain a Filippov solution until it reaches ${\bf 0} = (0,0)$. Then the remaining issue is uniqueness. We will use Proposition \ref{P4.1} to show that the Filippov solution exists uniquely if and only if $z(t) = (x(t),v(t))\neq {\bf 0}$ for any $t \in \bbr_+$. In what follows, we briefly preview the argument. We set
	\begin{align*}
	\begin{aligned}
		\mathcal{D}_1 :=\{ (x, v) \mid x<0 \}, \quad \mathcal{D}_2 :=\{(x, v) \mid x>0 \}, \quad S_Z:=\{(0,v) \mid v \in \bbr \}, \\
		Z\Big|_{\overline{\mathcal{D}}_1}: \{ (x, v) \mid x \leq 0 \} \to \bbr^2, \quad
		(x,v) \mapsto (v, -(\kappa_0 + \kappa_1) v -\kappa_2 x - \kappa_2 d^{\infty}), \\
		Z \Big|_{\overline{\mathcal{D}}_2}: \{ (x, v) \mid x \geq 0 \} \to \bbr^2, \quad
		(x,v) \mapsto (v, -(\kappa_0 + \kappa_1) v -\kappa_2 x + \kappa_2 d^{\infty}).
	\end{aligned}
	\end{align*}
	Then we can see that assumptions in Proposition \ref{P4.1} are fulfilled except the last one; both $Z \Big|_{\overline{\mathcal{D}}_1}$ and $Z \Big|_{\overline{\mathcal{D}}_2}$ are parallel to $S_Z$ at the origin. To resolve this issue, we modify the vector field $Z$ near the origin. More precisely, for $\varepsilon > 0$, we consider a vector field $Z^\varepsilon$ satisfying
	\begin{align}
	\begin{aligned} \label{New-D-3}
		Z^\varepsilon_{{|\mathcal{D}}_1}:
		(x,v) \mapsto 
		\left(v-\mathrm{sgn}(x)\times \varepsilon \varphi(\varepsilon^{-1}(x,v)), -(\kappa_0 + \kappa_1) v -\kappa_2 x - \kappa_2 d^{\infty}\right), \\
		Z^\varepsilon_{{|\mathcal{D}}_2}:
		(x,v) \mapsto 
		\left(v-\mathrm{sgn}(x)\times \varepsilon \varphi(\varepsilon^{-1}(x,v)), -(\kappa_0 + \kappa_1) v -\kappa_2 x + \kappa_2 d^{\infty}\right),
	\end{aligned}
	\end{align}
	where $\varphi:\bbr^2 \to \bbr$ is a smooth bump function satisfying 
	\[ \varphi \in C^{\infty}(\bbr^2), \quad  \varphi({\bf 0})=1 \quad \mbox{and} \quad  \mbox{spt}(\varphi) \subset B_1{({\bf 0})}. \]
	 Then $Z^\varepsilon$ satisfies the assumptions on Proposition \ref{P4.1}, and we will utilize $Z^\varepsilon$ to show that for any initial datum $(x^0,v^0)$, we have the following dichotomy: \vspace{.3cm}
	\begin{center}
		Either $(x(t^*),v(t^*))=(0,0)$ for some finite $t^*$, \vspace{.2cm} \\
		or there exists a unique global Filippov solution satisfying $\displaystyle \lim_{t \to \infty}|x(t)|=d^\infty$. \vspace{.2cm}
	\end{center}
	Note that, if the former holds, then the Cauchy problem \eqref{CSB1} is ill-posed, and there are infinitely many Filippov solutions (see Theorem \ref{T4.3}). For a Filippov solution $(x,v)$, we set $T$ to be the first hitting at  the origin:
		\begin{equation} \label{New-D-3}
			T := \inf\{ t \in \bbr_+ \mid (x(t), v(t)) = {\bf 0} \}.
		\end{equation}
\begin{lemma}\label{L}
	Suppose that $z^0 = (x^0,v^0) \neq {\bf 0}$ is an initial datum. Then $T \in (0,\infty]$ defined in \eqref{New-D-3} is well defined, and the Cauchy problem \eqref{CSB1} has a unique Filippov solution $z = z(t)$ on $[0,T)$.
	\end{lemma}
	\begin{proof} Consider the Cauchy problem to the modified system:
	\begin{equation}  \label{New-D-4}
	\begin{cases} 
	\displaystyle \dot{z}^\varepsilon(t) =Z^\varepsilon(z^\varepsilon(t)), \quad t > 0, \\
	\displaystyle z^\varepsilon(0) = z^0.
	\end{cases}
	\end{equation}
	Then, by Proposition \ref{P4.1}, the ODE system $\dot{z}(t)=Z^\varepsilon(z(t))$ has a unique Filippov solution $z^{\varepsilon}$. We take a sufficiently small $\varepsilon$ satisfying 
	\[ 0<\varepsilon<|z^0|, \] and let $T_\varepsilon$ be the first time hitting the $B_\varepsilon({\bf 0})$:
	\[
		T^\varepsilon := \inf \Big \{ t \in \bbr_+ \mid |z^\varepsilon(t)| = \varepsilon \Big \}.
	\]
	Then, since $Z$ and $Z^\varepsilon$ coincide on $\bbr^2 - B_\varepsilon(0,0)$, $\dot{z}(t)=Z(z(t))$ inherits this solution on $[0,T^\varepsilon]$ and this is a unique solution, because otherwise it contradicts the uniqueness of $z^\varepsilon$. As $T_\varepsilon$ is a decreasing function of $\varepsilon$, we can define its limit
	\[
		T:= \lim_{\varepsilon \searrow 0}T^\epsilon, \quad T \in (0,\infty].
	\]
	Note that, for $\varepsilon_1 > \varepsilon_2$, $Z^{\varepsilon_1}$ and $Z^{\varepsilon_2}$ coincide on $\bbr^2-B_{\varepsilon_1}({\bf 0})$, we have
	\[
		z^{\varepsilon_1}(t) = z^{\varepsilon_2}(t), \quad t \in [0, T^{\varepsilon_1}].
	\]
	Thus we may regard $z^{\varepsilon_2}([0,T^{\varepsilon_2}])$ as an extension of $z^{\varepsilon_1}([0,T^{\varepsilon_1}])$, and a solution $z(t)$ of \eqref{CSB1} exists uniquely on $[0,T)=\cup_{\epsilon}[0,T^\varepsilon]$. 
	Now for all $0 < \varepsilon \ll 1$, we have
	\[
		\inf_{t \in [0,T)}\mathrm{dist}(z(t),{\bf 0}) < \inf_{t \in [0,T^\varepsilon]}\mathrm{dist}(z(t),{\bf 0}) = \varepsilon.
	\]
	On the other hand, for any $t^* \in [0,T)$, there exists $0<\delta$ such that $t^* < T^\delta < T$ and therefore
	\[ 
		\inf_{t \in [0,t^*)}\mathrm{dist}(z(t), {\bf 0}) \geq \delta > 0.
	\]
	We combine the results altogether to get 
	\[
		\lim_{t \nearrow T}\mathrm{dist}(z(t),{\bf 0})
		= \inf_{t \in [0,T)}\mathrm{dist}(z(t), {\bf 0}) = 0.
	\]
	Therefore $T$ is the first hitting time of the origin. 
	\end{proof}	
Note that the concatenation of smooth solution mentioned above is in fact a unique Filippov solution. Now, are ready to show that desired distance will be achieved asymptotically despite collisions.
\begin{theorem}\label{T4.3}
	Let ${\bf 0} \neq (x^0,v^0) \in \bbr^2$ be a given initial data of \eqref{CSB1}. Then the following assertions hold.
	\begin{enumerate}
	\item System \eqref{CSB1} admits a unique global Filippov solution $(x,v)$ if and only if $(x(t),v(t))\neq {\bf 0}$ for arbitrary $t \in \bbr_+$. Otherwise, there are infinitely many Filippov solutions.
	\vspace{0.1cm}
	\item If $(x(t),v(t))\neq {\bf 0}$ for arbitrary $t \in \bbr_+$, then $\lim_{t \to \infty}|x(t)|=d^\infty$.
	\end{enumerate}
\end{theorem}
\begin{proof} 
	(1)~Depending on the coupling strengths, we consider two cases:
\[ \mbox{either}~(\kappa_0+\kappa_1)^2 \geq \kappa_2 \quad \mbox{or} \quad (\kappa_0+\kappa_1)^2 < \kappa_2. \]
\noindent $\bullet$~Case A:  Assume that  
	\[ (\kappa_0+\kappa_1)^2 \geq \kappa_2. \]
	If $x_0 \neq 0$, then the global solution is of the form
	\begin{align}\label{New-D-4.1}
		x=k_1\exp\left(-\frac{1}{2}t\left( (\kappa_0+\kappa_1) + K \right)\right)
		+k_2\exp\left(-\frac{1}{2}t\left( (\kappa_0+\kappa_1) - K \right)\right)
		+d_{12}^\infty\mathrm{sgn}(x^0),
	\end{align}
	where $k_1,k_2$ are determined from the initial data and $K:=\sqrt{(\kappa_0+\kappa_1)^2-4\kappa_2}$. Therefore, $|x-\mathrm{sgn}(x^0)d^\infty|$ decreases in time and $|x(t)|$ is always positive. If $x_0 = 0$, we replace $\mathrm{sgn}(x^0)$ by $\mathrm{sgn}(v^0)$ in \eqref{New-D-4.1} and we apply the same argument to see that $|x-\mathrm{sgn}(v^0)d^\infty|$ is decreasing and $|x(t)|$ is positive for all $t>0$
	\vspace{0.2cm}
	
\noindent $\bullet$~Case B: Assume that 
	\[ (\kappa_0+\kappa_1)^2 < \kappa_2. \]
	Suppose that $(x(\tilde{t}),v(\tilde{t}))=0$. We may set $\tilde{t}=0$, and we define functions $f^{\pm}$ as 
		\begin{align}\label{sol}
		\begin{aligned}
		f^{\pm}(x) &:= \mp e^{-\frac{1}{2}t(\kappa_0+\kappa_1)}
		d^\infty \sec \Big[ \arctan\left(-\frac{\kappa_0+\kappa_1}{\sqrt{\kappa_2 - (\kappa_0+\kappa_1)^2}}\right) \Big] \\
		&\hspace{1cm} \times \cos\Big[ \omega{t}+\arctan\left(-\frac{\kappa_0+\kappa_1}{\sqrt{\kappa_2 - (\kappa_0+\kappa_1)^2}}\right) \Big] \pm d^\infty.
		\end{aligned}
	\end{align}
	 Then for each $\tau \geq 0$,
	\[
		f_\tau^{\pm}(x):=
		\begin{cases}
			0 \quad &\text{if}  ~ x \leq \tau \\
			f^{\pm}(x) \quad &\text{if} ~ x > \tau
		\end{cases}
	\] 
	are solutions of \eqref{CSB1} on $[0,\tau+\varepsilon_\tau]$ for some $\varepsilon_\tau>0$. The remaining part is a direct consequence of Lemma \ref{L}. \newline
	
	\noindent (2)~We split its proof into two steps. \newline
	
	\noindent $\bullet$~Step A (finite number of collisions implies convergence): Suppose there exists only a finite number of collisions. If $x(t^*)=0$ for some finite $T$, we have \[ v(t^*) \neq 0. \]
	Consider a sequence of collision times $(t_n)$:
	\[
		(x(t_n),v(t_n))=(0,v(t_n))\neq(0,0), \quad t_1 < t_2 < t_3 < \cdots.
	\]
	Suppose that collision happens only finite times, and let $t_N$ be the last collision time. Then since $t_N$ is the last collision time, $x(t)$ satisfies
	\begin{align}\label{coll}
		\ddot{x}(t)=-(\kappa_0 + \kappa_1) \dot{x}(t) -\kappa_2 x(t) + \kappa_2 \mathrm{sgn}(v(t_N)) d^{\infty},
		\quad x(t)>0,
		\quad t>t_N.
	\end{align}
	Equation $\eqref{coll}_1$ represents a damped harmonic oscillator and it is well known that
	\[
		|x(t)-d^\infty(t)| \lesssim
		\begin{cases}
			\exp\left[ -\frac{1}{2}t\left( (\kappa_0+\kappa_1) - \sqrt{(\kappa_0+\kappa_1)^2-4\kappa_2} \right) \right ]  \quad &\text{if} \quad (\kappa_0+\kappa_1)^2 \geq 4\kappa_2, \\
			\exp\left[ -\frac{1}{2}t(\kappa_0+\kappa_1)\right]  \quad &\text{if} \quad (\kappa_0+\kappa_1)^2 < 4\kappa_2,
		\end{cases}
	\]
	for $t \geq t_N$. Therefore if collisions occur finitely many times, we have a desired convergence. \newline
	
	\noindent $\bullet$~Step B (the number of collisions is finite):~we prove that the number of collisions is finite. Suppose on the contrary that the collisions happen infinitely many times. Now, we claim that there exists $\delta>0$ such that for all $n$,
	\begin{equation} \label{New-D-5}
	t_{n+1}-t_n > \delta.
	\end{equation}
	If not, for all $\varepsilon > 0$, there exists $m$ such that 
	\[ t_{m+1}-t_m < \varepsilon. \]
	From Lemma \ref{L}, the Filippov solution is the continuous concatenation of solutions of either $\eqref{CSB2}_1$ or $\eqref{CSB2}_2$. From the dynamics of a damped harmonic oscillator, whether $(x((t_m,t_{m+1})),v((t_m,t_{m+1})))$ is in $\mathcal{D}_1$ or $\mathcal{D}_2$, we have 
	\[ |x(t_m+\varepsilon^*)|=d^\infty \quad \mbox{for some $0<\varepsilon^*<\varepsilon$}. \]
	Then since $x$ is differentiable in each time interval $(t_m,t_{m+1})$, the mean value theorem implies
	\begin{align*}
		d^\infty=|x(t_{m}+\varepsilon^*)-x(t_{m})|=\varepsilon^* \times |v(t_m^*)| < \varepsilon \times |v(t_m^*)|,
		\quad t_m^* \in (t_m,t_m+\varepsilon^*).
	\end{align*}
	Above, $\varepsilon^*$ can be taken arbitrarily small. This implies that $v$ can be arbitrary large:
	\begin{align}\label{New-D-6}
		\text{for all } 0<\varepsilon, \text{ there exists } t_m^* \in \bbr_+ \text{ such that }  ~ \frac{d^\infty}{\varepsilon} < |v(t_m^*)|.
	\end{align}
	On the other hand, since the continuity of solution yields 
	\[ \lim_{t \searrow t_n}E(t) = \lim_{t \nearrow t_n}E(t) \]
	and energy dissipates in each time interval, kinetic energy is bounded by $E(0)$, and so $v$ is bounded; this is contradictory to \eqref{New-D-6}, and the claim is proved.
	
	Note that total energy is decreasing and continuous on $\bbr_+$, and possibility of non-differentiability occurs only at each $t_n$. Thus $E$ admits a weak derivative and there exists 
	\[ \exists~~\lim_{t \to \infty}E(t) =: E^\infty. \]
	Therefore, we have
	\begin{align*}
		E^\infty-E(t_1)=\lim_{n \to \infty}E(t_n)-E(t_1)
		= -\lim_{n \to \infty}\frac{\kappa_0+\kappa_1}{2}\int_{t_1}^{t_{n}} v^2(t) dt,
	\end{align*}
	where we used the result of claim for the first equality, and Proposition \ref{CSEE} for the second equality. Since energy dissipation bounds both $x$ and $v$ on $\bbr_+$ and non-differentiability occurs only at each $t_n$, which is of measure zero, the result of Lemma \ref{L4.1} is still valid. As continuity of $v^2$ is guaranteed from Lemma \ref{L}, $v^2$ is uniformly continuous on $\bbr_+$. Hence we can apply BarBalat's Lemma to derive $v \to 0$. This implies $E_k \to 0$, and therefore 
\begin{align}\label{New-D-7}
	 \frac{\kappa_2}{4}(|x(t)|-d^\infty)^2=E_p \to E^\infty.
\end{align}
On the other hand, there exists $t_n^*\in (t_n,t_{n+1})$ satisfying
\[
	|x(t_n^*)|=d^\infty, \quad x(t_n)=0, \quad n \in \mathbb{N}, \quad \lim_{n \to \infty} t_n = \lim_{n \to \infty} t_n^* = \infty.
\]
This contradicts \eqref{New-D-7}, and the number of collisions is finite.
\end{proof}
	\section{Numerical simulations} \label{sec:5}
	\setcounter{equation}{0}
	In this section, we provide several numerical simulations for the second-order models in previous sections, and compare them with analytical results. Moreover, we also present several numerical simulations in relation with the convergence of relative distances toward the desired relative distances. 
	\subsection{Kuramoto ensemble} \label{sec:5.1}
	In this subsection, we present several numerical simulations for the Kuramoto model with the bonding force (KMBF) \eqref{Ku-SB} and compare them with those of the original Kuramoto (KM) \eqref{Ku-second} and the version with no Kuramoto term. We also check the consistency with the analytic results in Section \ref{sec:3} with simulation results. For all simulations, we choose $N = 10$ and use the 4th-order Runge-Kutta scheme. Initial data and system parameters are designed to satisfy the sufficient condition \eqref{C-1-3} for complete synchronization in Section \ref{sec:3}. Throughout this subsection, we set
	\[\Delta t=10^{-2},\quad t\in[0,5],\quad \nu_i=0,\quad\forall~i\in[10]. \]
	Recall the forcing terms in \eqref{Ku-SB}:
\begin{align*}
		&\underbrace{\frac{\kappa_0}{10} \sum_{j=1}^{10} \cos(\theta_j - \theta_i)(\omega_j - \omega_i)}_{\mbox{synchronizing force}}+ \underbrace{\frac{\kappa_1}{20}\sum_{j=1}^{10}(\dot{\theta}_j - \dot{\theta}_i)  + \frac{\kappa_2}{20}\sum_{j=1}^{10}(|\theta_j - \theta_i|- \theta^{\infty}_{ij})\textup{sgn}(\theta_j - \theta_i)}_{\mbox{bonding force}}.
\end{align*}
	Let $\Theta^{0}$ and $\Theta^*$ be the initial and target phase configurations:
	\begin{align*}
		\{\theta_i^0\}_{i=1}^{10} = \{&0.1979,\hspace{0.1cm} 0.2580, \hspace{0.1cm}0.2601, \hspace{0.1cm}0.4231,\hspace{0.1cm} 0.4635,\\
		&0.5011,\hspace{0.1cm} 0.5947,\hspace{0.1cm} 0.8710,\hspace{0.1cm} 0.9262,\hspace{0.1cm} 0.9722\},
	\end{align*}
	\begin{align*}
		\theta_i^*=
		\begin{cases}
			\displaystyle 3.5(i-1)^{\circ}\quad&\mbox{if}\quad1\le i\le3,\\
			\displaystyle 12.5^{\circ}+4(i-4)^{\circ}\quad&\mbox{if}\quad4\le i\le7,\\
			\displaystyle 40^{\circ}+5(i-8)^{\circ}\quad&\mbox{if}\quad8\le i\le10,
		\end{cases}
	\end{align*}
	Then, the matrix $ [\theta_{ij}^{\infty}]$ is determined by the target configuration $\Theta^*$ using the following relations:
	\[ |\theta_i^* - \theta_j^*| = \theta_{ij}^{\infty}, \quad i, j \in [10]. \]
	Lastly, $\{\omega_i\}$ is determined to make zero momentum when $(\kappa_0,\kappa_1,\kappa_2)=(1,5,10)$.
	In all simulations, we fix the initial configuration $\Theta^0$ and $[\theta_{ij}^{\infty}]$.  The first three sets of figures are concerned with the second-order Kuramoto model. \newline
In Figure 1, we compare the temporal evolution of phases and decay rates in the complete synchronization process. In Figure \ref{KMKMBF}(A), we can see that the trajectories of KM converge to the common phase, whereas the trajectories of KMBF tends to the preassigned target configuration $\Theta^*$. Of course, the rigorous justification for this convergence has not been verified. In Figure \ref{KMKMBF}(B), we can see that the decay rates for complete synchronization are at least exponential, and complete synchronization for KMBF seems to occur faster than that of KM. This is due to the bonding control so that aggregated configuration tends to the target configuration $\Theta^*$ much faster than the original KM ensemble. Note that the analytical result in Theorem \ref{T3.2} provide a zero convergence of relative frequencies without any decay rate. 
	\begin{figure}
		\begin{subfigure}[t]{3 in}
			\epsfig{file=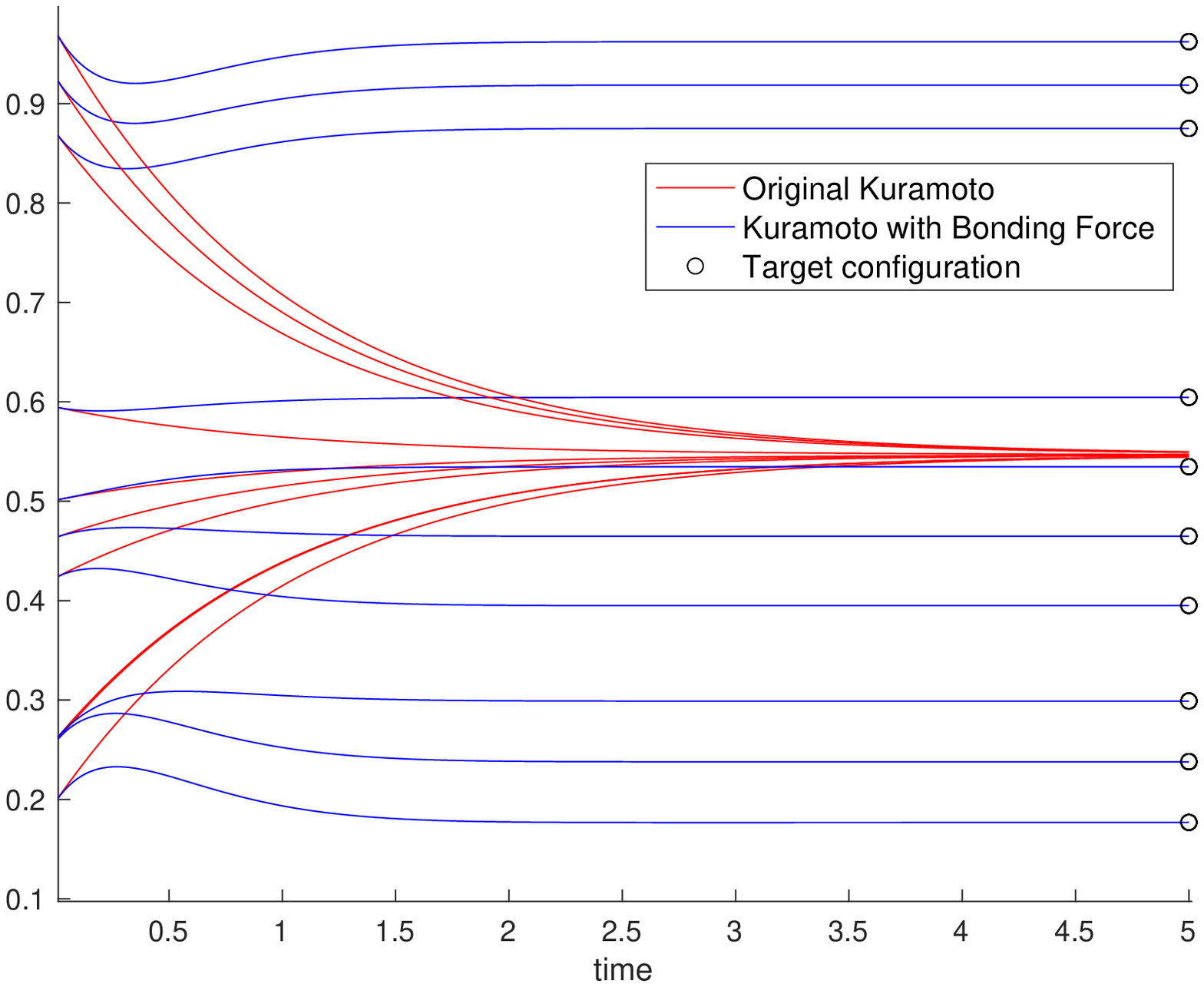, angle=0, width=6.8cm}
			\caption{Evolution of phases}
			\label{KMKMBFA}
		\end{subfigure}
		\begin{subfigure}[t]{3 in}
			\epsfig{file=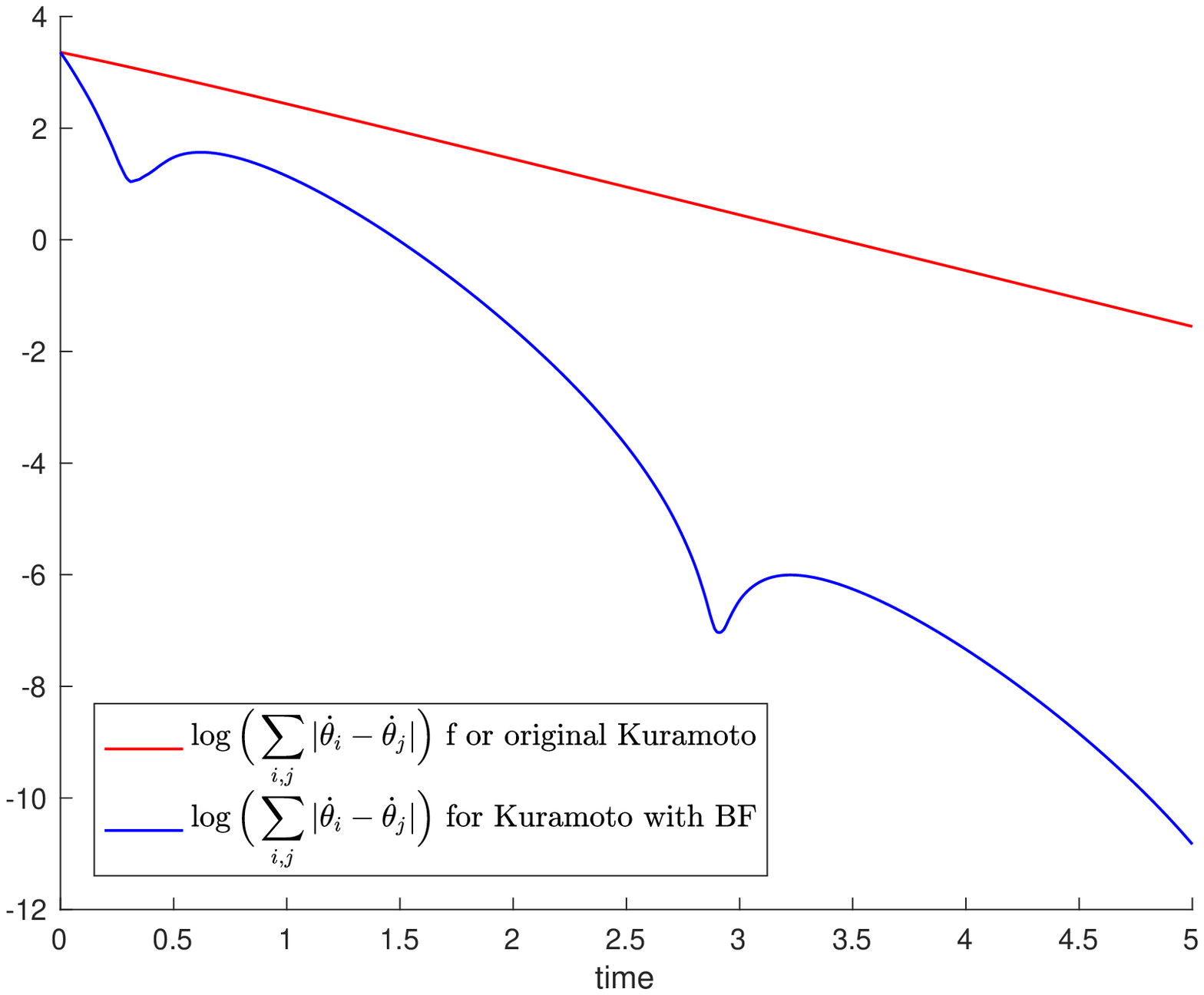, angle=0, width=6.8cm}
			\caption{Convergence rate}
			\label{KMKMBFB}
		\end{subfigure}
		\caption{Comparison of relaxation processes for KM and KMBF}
		\label{KMKMBF}
	\end{figure}
\vspace{0.2cm}	

 In Figure 2, we compare the trajectories of the KMBF with two sets of coupling strengths to observe the impact of Kuramoto term:
	\[ (\kappa_0, \kappa_1, \kappa_2) = 
	\begin{cases}
		(1, 5, 10) \quad & \mbox{Left figure}, \\
		(0, 5, 10) \quad & \mbox{Right figure}.
	\end{cases}
	\]
	In both cases, phase trajectories tend to the target phase configuration $\Theta^*$ as time goes on. The synchronization force $(\kappa_0 > 0$) can affect the trajectories in initial layer, in this case it seems the Kuramoto term makes some attraction force between particles at the beginning, but it does not affect to the resulting target phase configuration. Of course, this obvious fact is not yet proved. 
	\begin{figure}
	\begin{subfigure}[t]{3 in}
			\epsfig{file=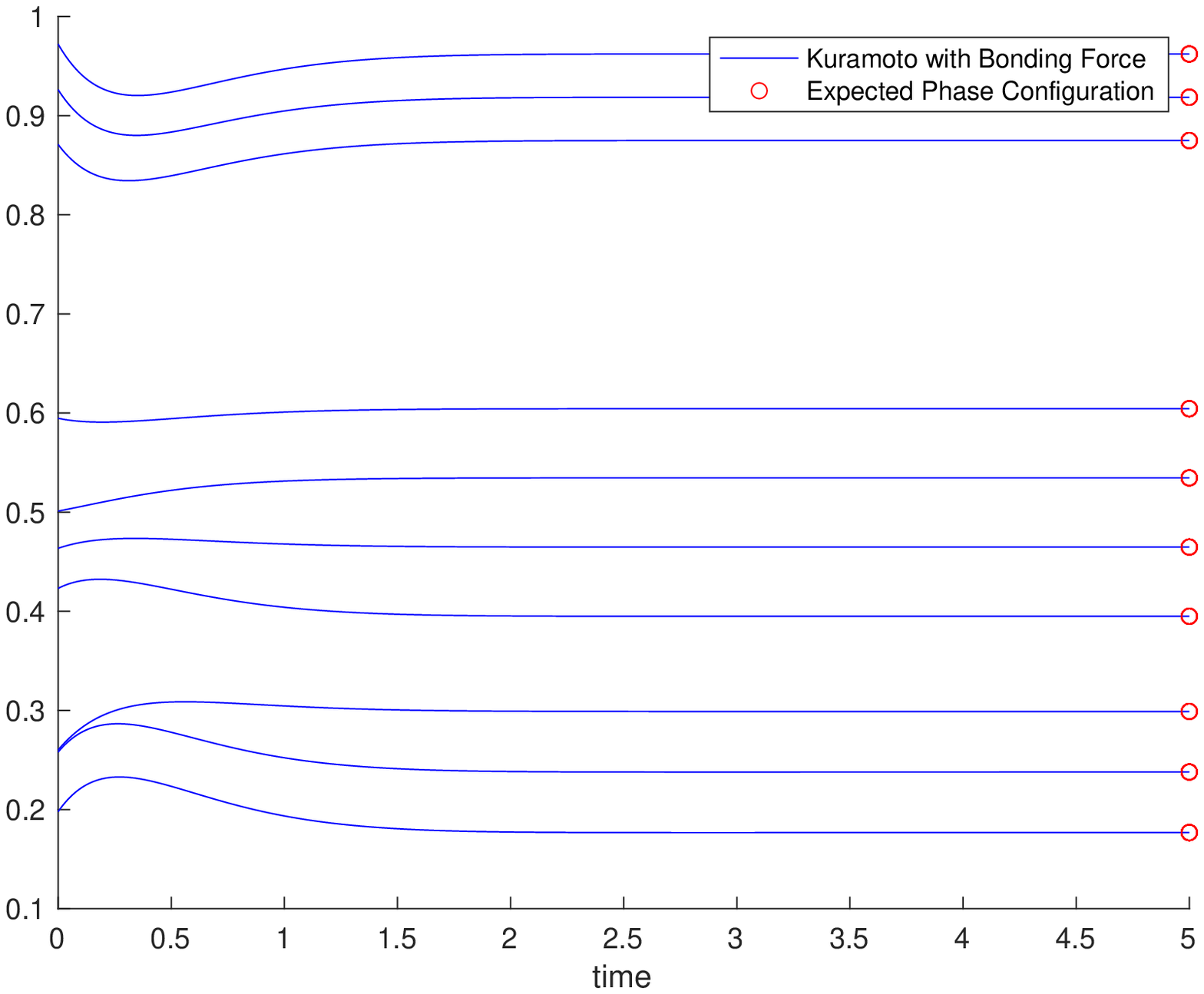, angle=0, width=6.8cm}
			\caption{$(\kappa_0, \kappa_1, \kappa_2) = (1,5,10)$}
		\end{subfigure}
		\begin{subfigure}[t]{3 in}
			\epsfig{file=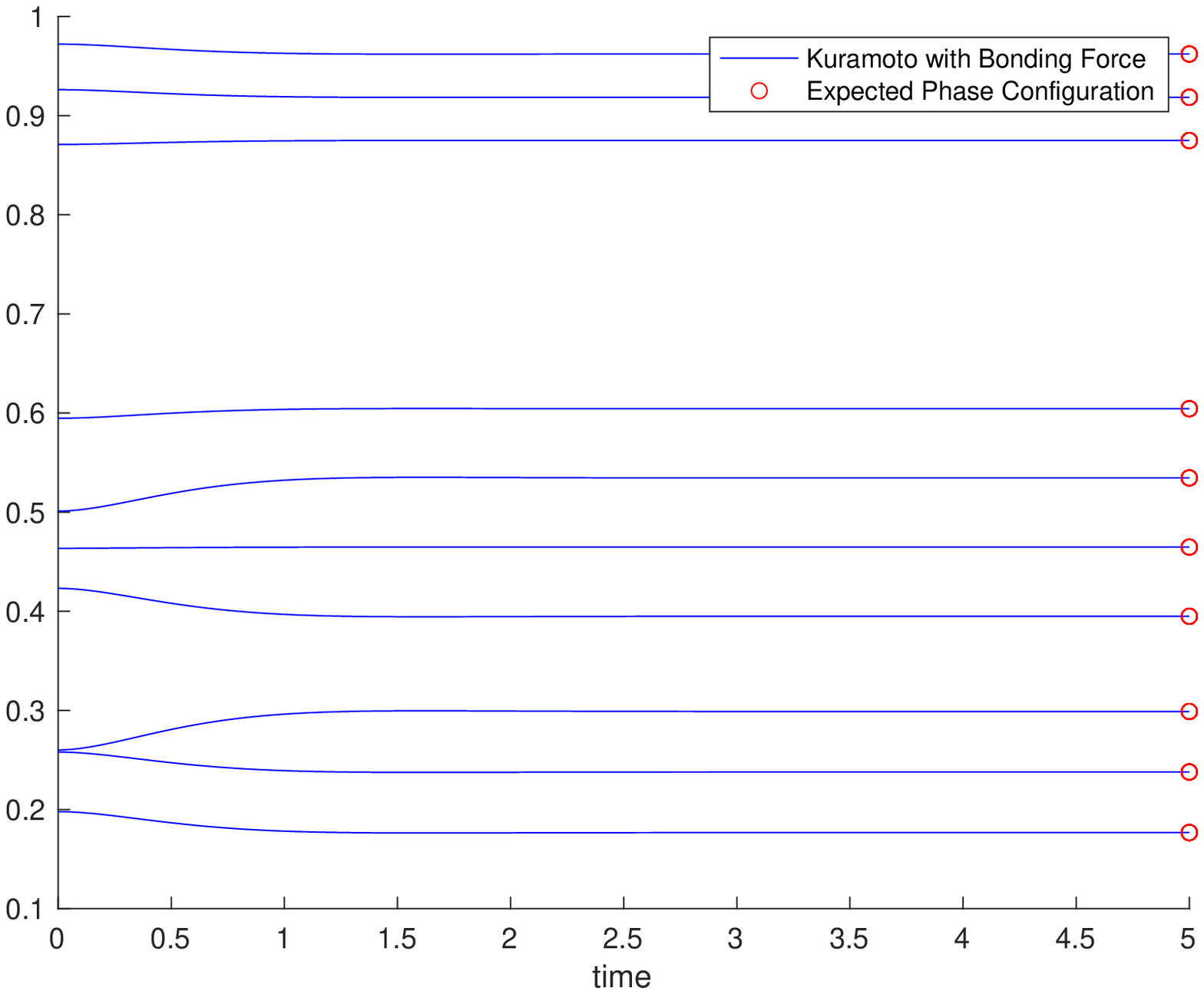, angle=0, width=6.8cm}
			\caption{$(\kappa_0, \kappa_1, \kappa_2) = (0,5,10)$}
			\label{PDW}
		\end{subfigure}
		\caption{Convergence to target configuration}\label{KCONF}
	\end{figure}	
	\vspace{0.2cm}
	
In Figure 3, we see the temporal evolutions of kinetic, potential and total energies for Kuramoto ensemble in a bonding force field. As analytically shown in Proposition \ref{KEE}, total energy monotonically decreases to zero for well-prepared initial data and system parameters, but kinetic and potential energies decay to zero without a monotonicity.  	
	\begin{figure}
	\begin{subfigure}[t]{3 in}
			\epsfig{file=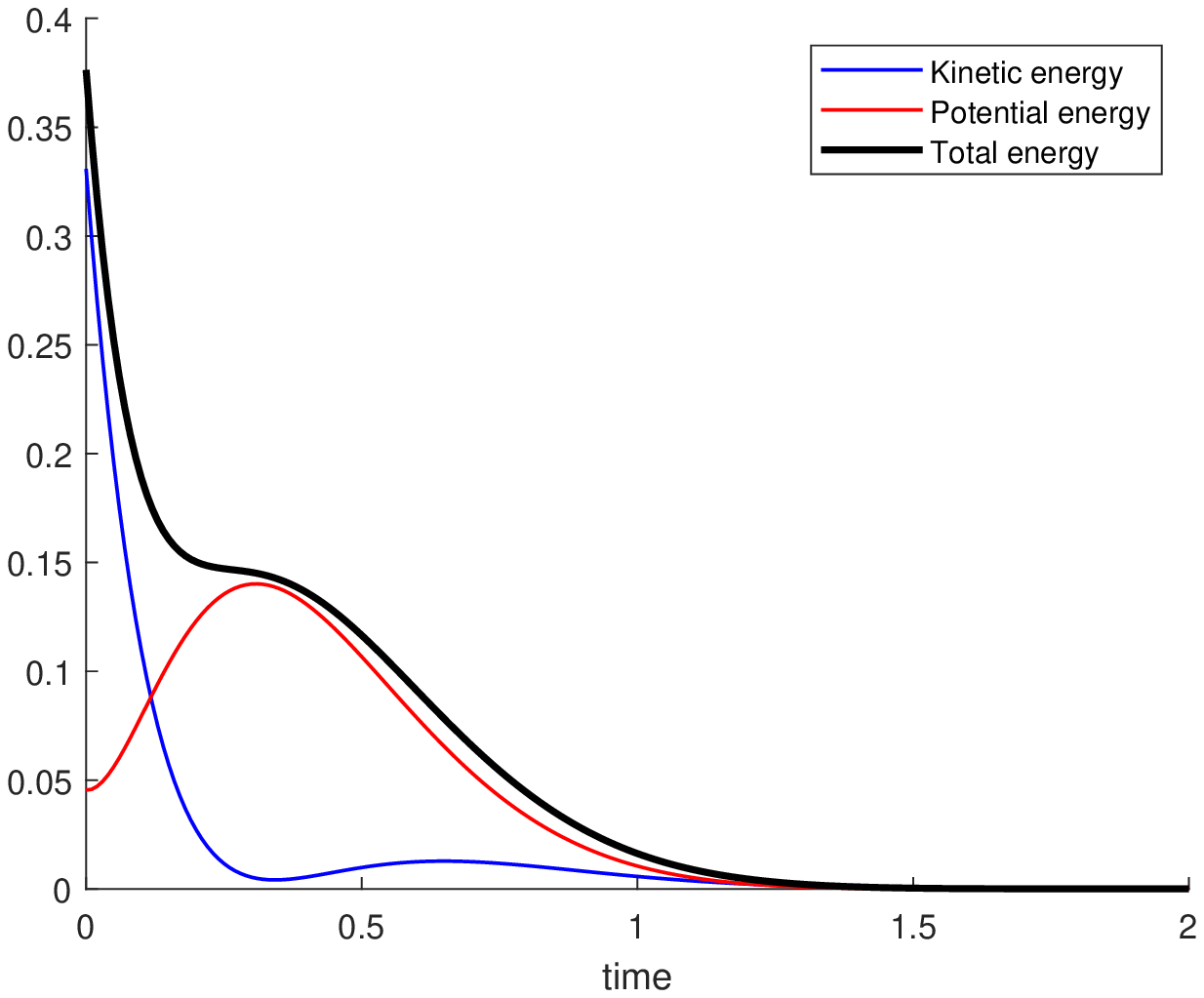, angle=0, width=6.8cm}
			\caption{$(\kappa_0, \kappa_1, \kappa_2) = (0,5,10)$}
		\end{subfigure}
		\begin{subfigure}[t]{3 in}
			\epsfig{file=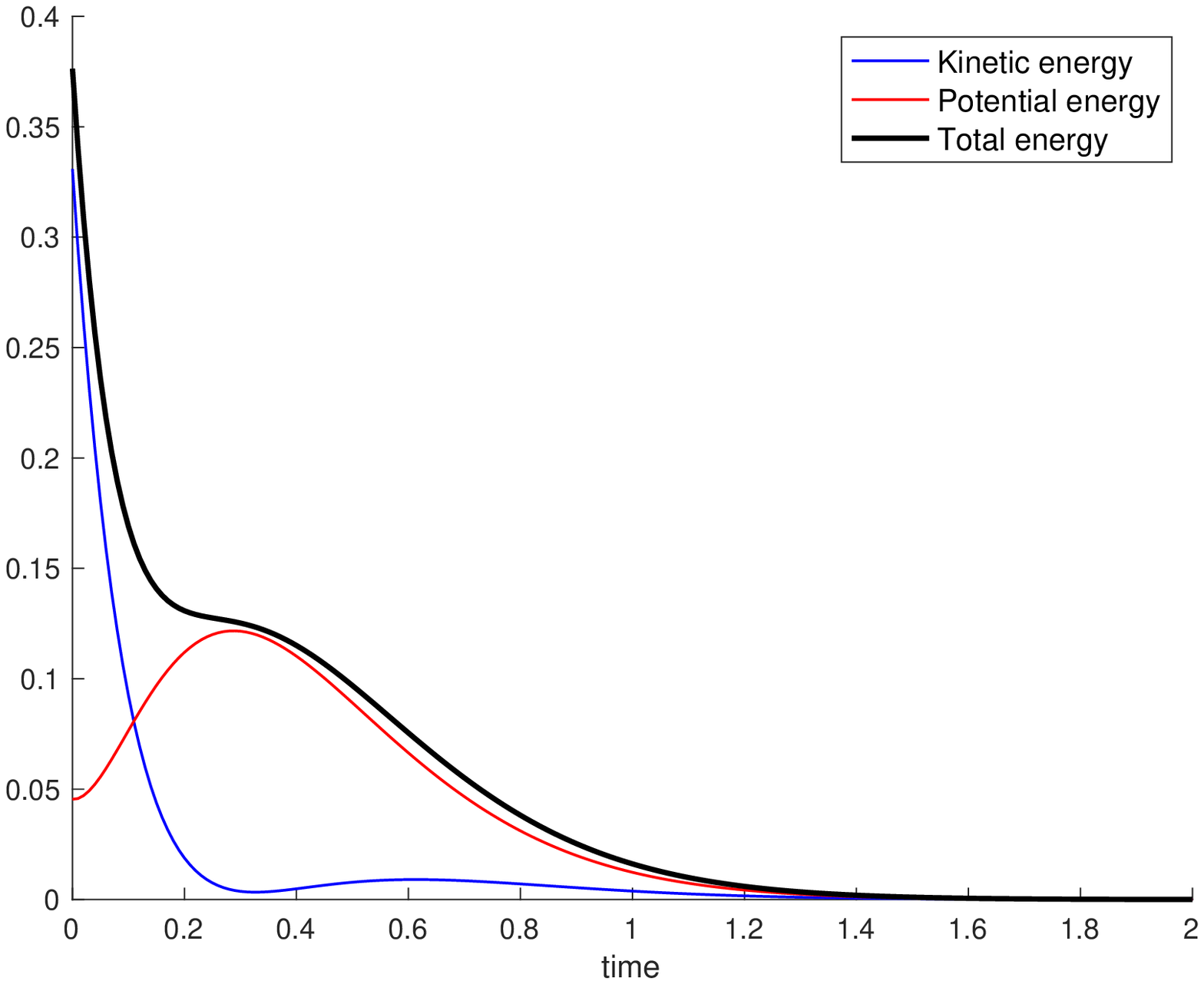, angle=0, width=6.8cm}
			\caption{$(\kappa_0, \kappa_1, \kappa_2) = (1,5,10)$}
			\label{PDW}
		\end{subfigure}
		\caption{Temporal evolution of energies} \label{1510energy}
	\end{figure}	

	\subsection{Cucker-Smale ensemble}
	\hspace{.1cm} In this subsection, we provide various numerical results for one and two dimensional Cucker-Smale system with the bonding force (CSBF) \eqref{CSB}. For all simulations, we use $N = 10$ and the 4th-order Runge-Kutta scheme. Initial data and system parameters are designed to satisfy the sufficient condition \eqref{D-3-1} for a global flocking in Section \ref{sec:4}. We also set
	\begin{align}\label{setting_CS_numeric}
		\Delta t=10^{-2},\quad t\in[0,10],\quad\psi({\bx})=\frac{1}{1+\|{\bx}\|}.
	\end{align}
	Recall the forcing terms \eqref{CSB}:
	\begin{align*}
		&\frac{\kappa_0}{10}\sum_{j=1}^{10}\psi(\|\bx_i-\bx_j\|)\left(\bv_j- \bv_i\right)+\frac{1}{10}\sum_{i\ne j}^{10}
		\left[ \kappa_1 \frac{\langle \bv_j-\bv_i,\bx_j-\bx_i\rangle}{\|\bx_i-\bx_j\|^2}+ \kappa_2\frac{(\|\bx_i-\bx_j\|- d^{\infty}_{ij})}{\|\bx_i-\bx_j\|} \right] (\bx_j-\bx_i).
	\end{align*}
	\hspace{.1cm} We set initial configuration as 
	\begin{table}[H]
		\begin{center}
			\begin{tabular}{c||c|c}
				\hline
				$(\bx_1^0,\bv_1^0)$ & (2.9415, 1.0133) & (0.0100, -0.1275) \\
				\hline
				$(\bx_2^0,\bv_2^0)$ & (-0.1868, 3.0893) & (0.0874, 0.2318) \\
				\hline
				$(\bx_3^0,\bv_3^0)$ & (-2.8378, 0.6900) & (0.0192, 0.1613) \\
				\hline
				$(\bx_4^0,\bv_4^0)$ & (-1.8895, -2.4844) & (0.0450, 0.0151) \\
				\hline
				$(\bx_5^0,\bv_5^0)$ & (1.9088, -2.3172) & (0.0099, -0.0733) \\
				\hline
				$(\bx_6^0,\bv_6^0)$ & (0.4133, 0.9212) & (0.0301, -0.1290) \\
				\hline
				$(\bx_7^0,\bv_7^0)$ & (-0.4425, 0.7271) & (-0.1415, -0.1233) \\
				\hline
				$(\bx_8^0,\bv_8^0)$ & (-0.8685, -0.5283) & (-0.2134, 0.1876) \\
				\hline
				$(\bx_9^0,\bv_9^0)$ & (-0.0589, -0.9098) & (0.0256, -0.0149) \\
				\hline
				$(\bx_{10}^0,\bv_{10}^0)$ & (1.0304, -0.2013) & (0.1278, -0.1280) \\
				\hline
			\end{tabular}
		\end{center}
		\label{ID}
	\end{table}
	Note that the initial data in the table are chosen to satisfy zero sum conditions:
	\[\sum_{i=1}^{10}{\bx}_i^0=0,\quad\mbox{and}\quad\sum_{i=1}^{10}{\bv}_i^0=0.\]
	The matrix $[d_{ij}^\infty]$ is determined by the relative distances among given 10 points $\{{\bf s}_i\}_{i=1}^{10}$ which are called the target configuration:
	\begin{align*}
		{\bf s}_i =
		\begin{cases}
			\frac{3}{2}\left(\cos(18+72(i-1))^{\circ},\sin(18+72(i-1))^{\circ}\right),\quad\mbox{if}\quad1 \leq i \leq 5, \\
			\frac{1}{2}\left(\cos(54+72(i-1))^{\circ},\sin(54+72(i-1))^{\circ}\right),\quad\mbox{if}\quad \leq i \leq 10.
		\end{cases}
	\end{align*}
In all the simulations, we fix the initial configuration $(X^0, V^0)$ and the matrix  $[d_{ij}^\infty]$. For the spatial pattern configuration, if  $[d_{ij}^\infty]$ is randomly given, the existence of particles satisfying the distances  $[d_{ij}^\infty]$ is not guaranteed.

In Figure \ref{1510CSenergy}, we can see the temporal evolutions of kinetic, potential and total energies for two different set of coupling strengths:
\[ (\kappa_0,\kappa_1,\kappa_2):~(0,5,10) \quad (1,5,10). \]
Note that in (A), the kinetic energy decays to zero asymptotically for a solution with zero total momentum which is consistent with Theorem \ref{T4.2} (ii). In addition, the potential energy also decrease to zero asymptotically and this means that all particles maintain the expected distances. In (B), the story of potential energy is the same with that of (A). However, the kinetic energy converges to nonzero implying that it does not exhibit asymptotic flocking. From this numeric simulations, we can derive that the condition of strictly positive $\kappa_0$ in \eqref{D-0-0} is tightened.\newline
	\begin{figure}
	\begin{subfigure}[t]{3 in}
			\epsfig{file=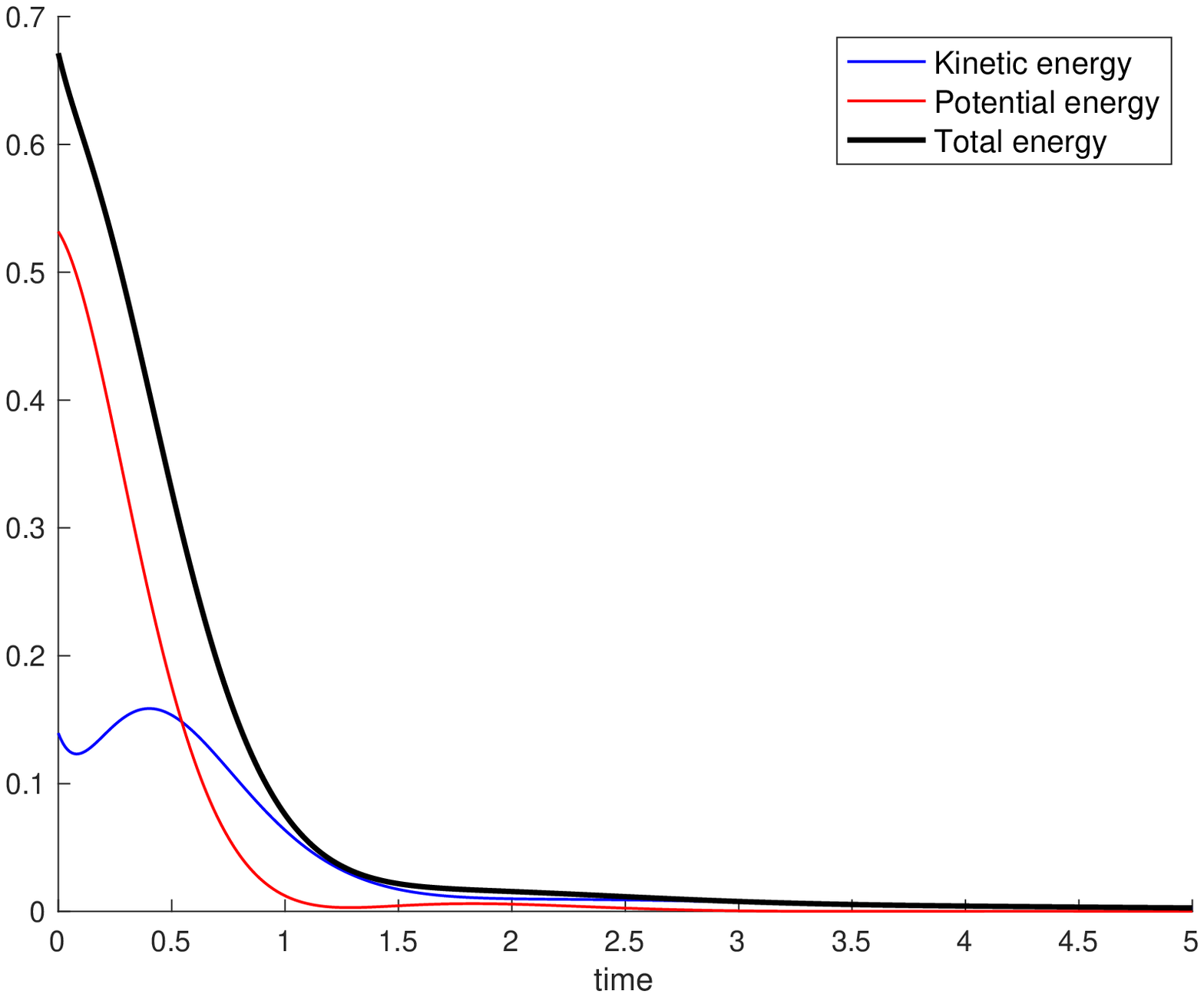, angle=0, width=6.8cm}
			\caption{$(\kappa_0, \kappa_1, \kappa_2) = (1,5,10)$}
		\end{subfigure}
		\begin{subfigure}[t]{3 in}
			\epsfig{file=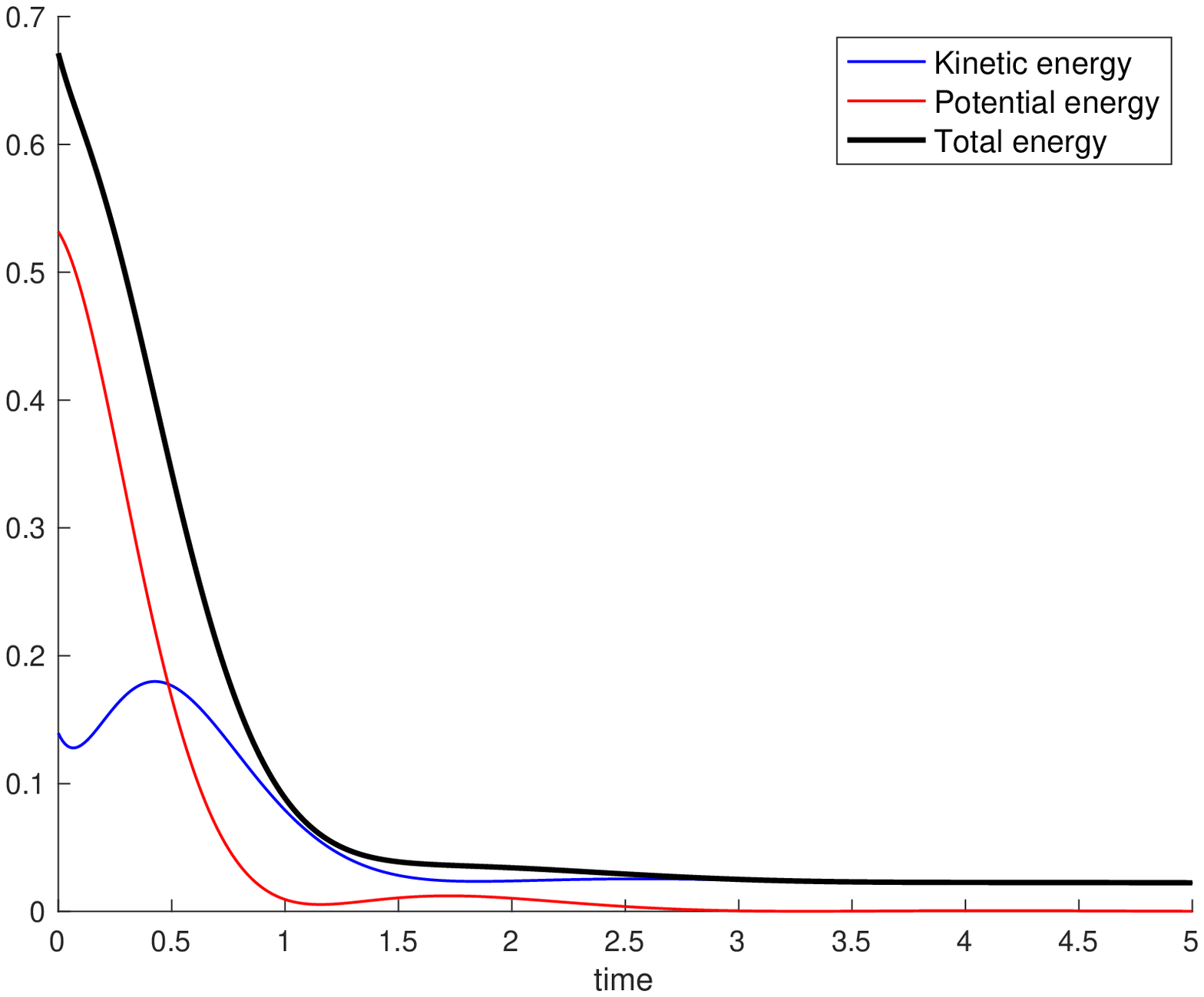, angle=0, width=6.8cm}
			\caption{$(\kappa_0, \kappa_1, \kappa_2) = (0,5,10)$}
			\label{PDW}
		\end{subfigure}
		\caption{Temporal evolution of energies for $d = 2$} \label{1510CSenergy}
	\end{figure}
 
Finally, we consider the convergence of relative distances for the Cucker-Smale system with bonding force(CSBF) on the line $d=1$. In this situation, the convergence of relative distances to the preassinged set $[d_{ij}^{\infty}]$ can be observed numerically.
	
In Figure \ref{CS_1d_energy}, we focus on the rate of reduction of kinetic, potential and total energies, respectively. We maintain numerical settings \eqref{setting_CS_numeric} and assume zero total momentum as well. Likewise in the Kuramoto system with a bonding force which is 1-dimensional system, 1-dimensional CSBF also exhibits precise configuration of particles compared to expected distances. For simulations, we adopted $(\kappa_0,\kappa_1,\kappa_2)=(1,1,40)$ and the initial and target position configurations as follows:
	\begin{align*}
		\{x_i^0\}_{i=1}^{10}=\{&-29.5926,\hspace{0.1cm} -16.5471, \hspace{0.1cm}-8.9365, \hspace{0.1cm}-3.5433,\hspace{0.1cm} -0.6838,\\
		&1.0488,\hspace{0.1cm} 4.1392,\hspace{0.1cm} 9.2734,\hspace{0.1cm} 17.4788,\hspace{0.1cm} 30.4824\},
	\end{align*}
		\[\{x_i^*\}_{i=1}^{10}=\{-30,\hspace{0.1cm} -17, \hspace{0.1cm}-9, \hspace{0.1cm}-4,\hspace{0.1cm} -1,\hspace{0.1cm} 1,\hspace{0.1cm} 4,\hspace{0.1cm} 9,\hspace{0.1cm} 17,\hspace{0.1cm} 30\}.\]
This simulation has the energy configuration as in Figure \ref{CS_1d_energy}.
	\begin{figure}
		\centering
		\begin{subfigure}[c]{3 in}
			\epsfig{file=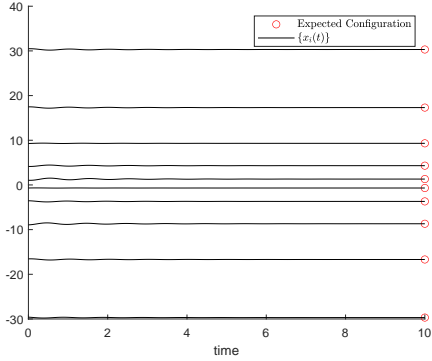, angle=0, width=6.8cm}
			\caption{ Temporal evolution of $\{x_i\}_{i=1}^{10}$}
		\end{subfigure}
		\begin{subfigure}[c]{3 in}
			\epsfig{file=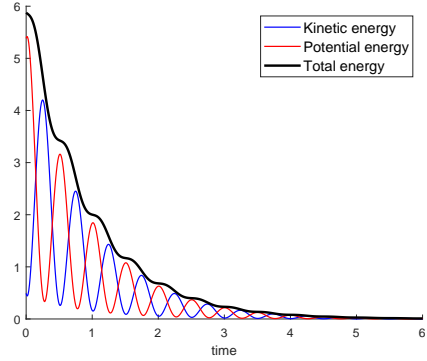, angle=0, width=6.8cm}
			\caption{Temporal evolution of energies}
		\end{subfigure}	
		\begin{subfigure}[c]{3 in}
			\epsfig{file=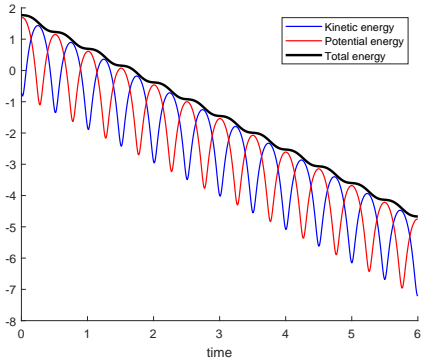, angle=0, width=6.8cm}
			\caption{Decay rates of energeis}
		\end{subfigure}	
		\begin{subfigure}[c]{3 in}
			\epsfig{file=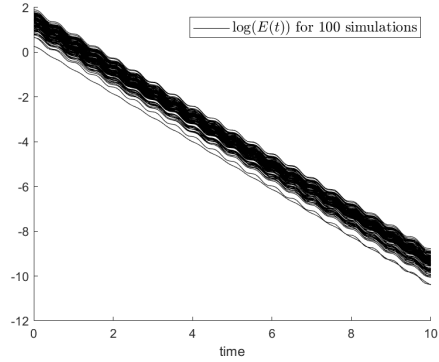, width=6.8cm}
			\caption{Decay of total energy for 100 simulations}      
		\end{subfigure}
		\caption{Convergence of relative distances for $d = 1$}
		\label{CS_1d_energy}
	\end{figure}
Note that the total energy monotonically decrease in (B) whereas kinetic and potential decay to zero with infinite number of oscillations. To investigate the flows of decresing rate and oscillation amplitude, we measure them in a logarithm scale as in (C). With this same context, we conducted 100 simulations where each has only difference in initial data of position and velocity. The simulation result (D) illustrates that the amplitude could be different depending on initial data. However, the exponential decay rate seems to be the same irrespective of initial data.

		\section{Conclusion} \label{sec:6}
	In this paper, we have introduced two second-order nonlinear consensus models with an inter-particle bonding force, namely the ``{\it Kuramoto model with a bonding force}" and the ``{\it Cucker-Smale model with a bonding force}". These proposed models contain singular terms in the bonding force which are singular at the instant in which some state variables coincide with. We simply call these situations as collisions. Thus, if we work in a classical framework of well-posedness given by the Cauchy-Lipschitz theory, we can show the nonexistence of finite collisions and exhibit asymptotic consensus estimates under some conditions on system parameters and initial data. In this direction, we provided several explicit analytical frameworks for collision avoidance and collective dynamics of the proposed models. The proposed frameworks are formulated in terms of system parameters and initial data.  Of course, there any many interesting issues that we did not touch in the current work. To name a few, we first consider the convergence issue of relative states. When the coupling strengths are sufficiently small so that the proposed models can be regarded as the perturbations of the corresponding linear flows, finite-time collisions can emerge. In this case, we may work with a Filippov framework which is beyond the scope of the current work. However, for a two-particle system on the real line, we show that the relative states for Filippov solution tend to the preassigned values even if finite-time collisions are present with the help of Filippov theory. Thus, the generalization of this special case to a one-dimensional setting for a many-body system will be an interesting open problem. The  extension of current work to the relativistic and manifold setting as in \cite{H-K-Rug-1,H-K-S} will be another interesting direction for a futrue work.	
	\newpage
	
	\appendix
	
	\section{Bonding feedback control for the Cucker-Smale model} \label{App-A}
	\setcounter{equation}{0}
	In this appendix, we present a formal heuristic derivation of \eqref{CSB} on the Euclidean space $\mathbb{R}^d$ following the arguments in \cite{P-K-H}. First, we begin with the Cucker-Smale model with all-to-all coupling and forcing term:
	\begin{equation}  \label{AB-0}
		\begin{cases}
			\displaystyle {\dot{\bx}_i} =\bv_i,\quad t >0,~ i \in [N],\\
			\displaystyle {\dot{\bv}_i} =\frac{\kappa_0}{N}\sum_{j=1}^{N} \psi(\|\bx_j - \bx_i\|) \left(\bv_j- \bv_i\right)  + \bbf_{i}.
		\end{cases}
	\end{equation}
	Here, $\bbf_{ji}$ is a bonding force exerted on the test $i$ particle by field particle $j$ and bonding force $\bbf_i = \frac{1}{N} \sum_{j= 1}^{N}  \bbf_{ji}$ is the bonding force exerted on the test $i$ particle by all the  field particles.  To focus on the design of a bonding force $\bbf_{ji}$, we set the unit tangent vector $\bn_{ij}$ in the direction of $\bx_i - \bx_j$: 
	\[ \bn_{ij} := \frac{\bx_i-\bx_j}{\|\bx_i-\bx_j\|}=\frac{\bx_i-\bx_j}{r_{ij}}. \] 
	Recall that our goal is to design a force:
	\begin{equation} \label{AB-1-1}
		\bbf_{ji} = f_{ji} \bn_{ij}
	\end{equation}
	such that in the absence of other particles except $i$-th and $j$-th particle, the relative velocity is exactly zero
	\[  \|\bv_j-\bv_i\| = 0, \]
	at the instant in which $r_{ij} =d_{ij}^{\infty}$. \newline
	
	Next, we define the deviation functional $e_{ij}$ and velocity component $v_{i,j}$ of $\bv_i$ along the direction of $\bn_{ij}$:
	\begin{equation} \label{AB-2}
		e_{ij}:=r_{ij} - d^{\infty}_{ij}, \quad v_{i,j} := \langle \bv_i, \bn_{ij} \rangle, \quad v_{j,i} := \langle \bv_j, \bn_{ji} \rangle. 
	\end{equation}
	Moreover, we design the pairwise inter-particle bonding force magnitude $f_{ij}$ to satisfy
	\begin{align}\label{AB-1}
		\dot{v}_{i,j}:=f_{ji},\quad \dot{v}_{j,i}:=f_{ij},\quad f_{ji}=f_{ij}.
	\end{align}
	It follows from \eqref{AB-2} and \eqref{AB-1} that 
	\begin{align}
	\begin{aligned} \label{AB-3}
			\dot{e}_{ij} &= \dot{r}_{ij}= \langle \bv_i - \bv_j, \bn_{ij} \rangle  =  \langle \bv_i, \bn_{ij} \rangle - \langle  \bv_j, \bn_{ij} \rangle =  \langle \bv_i, \bn_{ij} \rangle  +  \langle  \bv_j, \bn_{ji} \rangle 
			= v_{i,j} + v_{j,i},  \\
			\ddot{e}_{ij} &= \dot{v}_{i,j} + \dot{v}_{j,i}=  f_{ji}  + f_{ij} = 
			2f_{ji},
	\end{aligned}
	\end{align}	
	where we used the relations:
	\[ f_{ij} = f_{ji}, \quad -\bn_{ij} = \bn_{ji}. \]
	Now, we choose $f_{ji}$ in \eqref{AB-3} as follows:
	\begin{equation} \label{AB-4}
	  f_{ji} = -\kappa_1 {\dot e}_{ij} -\kappa_2 e_{ij},
	\end{equation}
	where $\kappa_1$ and $\kappa_2$ are positive constants. Thus, $e_{ij}$ satisfies 
	\[
	\ddot{e}_{ij}+ 2\kappa_1\dot{e}_{ij}+ 2\kappa_2{e}_{ij}=0.
	\]
	This yields that  $e_{ij}$ and ${\dot e}_{ij}$ tend to zero asymptotically:
	\[ \lim_{t \to \infty} e_{ij}(t) = 0 \quad \mbox{and} \quad  \lim_{t \to \infty} {\dot e}_{ij}(t) = 0. \]
	Finally, it follows from \eqref{AB-2}, \eqref{AB-3} and \eqref{AB-4} that 
	\begin{align}
	\begin{aligned} \label{AB-4-0}
	 f_{ji}  &= - \kappa_1 {\dot e}_{ij} - \kappa_2 e_{ij} = -\kappa_1 \langle \bv_i - \bv_j, \bn_{ij} \rangle  - \kappa_2 (r_{ij} - d^{\infty}_{ij})  \\
	 &=  -\kappa_1 \langle \bv_j - \bv_i, \bn_{ji} \rangle  - \kappa_2 (r_{ij} - d^{\infty}_{ij}) \\
	 &=  -\kappa_1 \Big \langle \bv_j -\bv_i, \frac{\bx_j-\bx_i}{r_{ji}} \Big \rangle - \kappa_2 (r_{ij} - d^{\infty}_{ij}).
	\end{aligned}
	 \end{align}
	Again, by \eqref{AB-1-1}, one has 
	\begin{equation} \label{AB-5}
		\bbf_{ji} = f_{ji} \bn_{ij} =  - f_{ji} \bn_{ji} = \kappa_1 \Big \langle \bv_j-\bv_i, \frac{\bx_j-\bx_i}{r_{ji}} \Big \rangle \frac{(\bx_j - \bx_i)}{r_{ji}} +  \kappa_2 (r_{ij} - d_{ij}^{\infty} ) \frac{(\bx_j - \bx_i)}{r_{ji}}.
	\end{equation}	
	Hence, we combine \eqref{AB-0} and \eqref{AB-5} to find 
	\begin{equation}
		\begin{cases} \label{CSB'}
			\displaystyle {\dot{\bx}_i} =\bv_i,\quad t>0,\quad  i \in [N],\\
			\displaystyle {\dot{\bv}_i} =\frac{\kappa_0}{N}\sum_{j=1}^{N}\psi(r_{ij})\left(\bv_j- \bv_i \right)\\
			\hspace{0.5cm}+\displaystyle\frac{\kappa_1}{N}\sum_{j \neq i}^N \Big \langle \bv_j-\bv_i, \frac{\bx_j-\bx_i}{r_{ji}} \Big \rangle \frac{(\bx_j-\bx_i)}{r_{ji}}  + \frac{\kappa_2}{N}\sum_{j \neq i}^N 
			(r_{ij} - d^{\infty}_{ij}) \frac{(\bx_j-\bx_i)}{r_{ji}}.
		\end{cases}
	\end{equation}

	\section{Bonding feedback control for the Kuramoto model} \label{App-B}
	\setcounter{equation}{0}
	In this appendix, we discuss a heuristic derivation of inter-particle bonding force for the second-order Kuramoto model following the same strategy in Appendix A:
	\begin{equation} \label{AA-0-0}
\begin{cases}
\displaystyle {\dot \theta}_i = \omega_i, \quad t > 0, \quad  i \in [N], \\
\displaystyle {\dot \omega}_i =  \frac{\kappa_0}{N} \sum_{j=1}^{N} \cos(\theta_j - \theta_i)(\omega_j - \omega_i) + \frac{1}{N} \sum_{j \neq i}  \bbf_{ji}.
\end{cases}
\end{equation}

  \subsection{Differential geometry for the unit circle} In this part, we briefly discuss minimum materials regarding the differential geometry of the unit circle ${\mathbb S}^1$ which can be regarded  as  the one-dimensional Riemannian manifold embedded in ${\mathbb R}^2$. Let $\bx \in {\mathbb S}^1$.  Then, the exponential map $\exp_{\bx} : T_{\bx} {\mathbb S}^1 \to {\mathbb S}^1$ is defined by $\exp_{\bx} \bv = \gamma(1),$
	where $\bv \in T_{\bx} {\mathbb S}^1$ and $\gamma(t) : [0, 1] \to {\mathbb S}^1$ is a geodesic on $\bbs^1$ satisfying 
	\[ \gamma(0) = \bx \quad \mbox{and} \quad  {\dot \gamma}(0) = \bv. \]
	Let ${\tilde {\mathcal I}}_{\bx} \subset  T_{\bx} {\mathbb S}^1$ be the maximal open set on which $\exp_{\bx}$ is a diffeomorphism
	and define the interior set as ${\mathcal I}_{\bx} := \exp_{\bx} ({\tilde {\mathcal I}}_{\bx}) \subset {\mathbb S}^1$. The exponential map is invertible on ${\tilde{\mathcal I}}_{\bx}$, hence we define its inverse as the logarithm map $\log_{\bx} := \exp_{\bx}^{-1} : {\mathcal I}_{\bx} \to T_{\bx} {\mathbb S}^1$:
	\[ \log_{ \bx} { \by} = {\dot {\gamma}}(0), \quad  { \by} \in {\mathcal I}_{  \bx}. \]
	Here, $\gamma: [0, 1] \to {\mathbb S}^1$ is the length minimizing geodesic satisfying $\gamma(0) = \bx$ and $\gamma(1) = \by.$
	The bonding control term $f_{ij}$ acts along the unit tangent vector 
	\[ \frac{\log_{\bx_i}\bx_j}{d(\bx_i,\bx_j)} =  \frac{\log_{\bx_i}\bx_j}{d_{ij}},\quad \bx_i\neq \pm\bx_j,\] where $d_{ij} := d(\bx_i,\bx_j)$ is a length-minimizing geodesic distance between $\bx_i$ and $\bx_j$. Note that 
	the term $ \frac{\log_{\bx_j}\bx_i}{d(\bx_j,\bx_i)}$ will play the same role of $\bn_{ij}$ in Appendix A. \newline
	
	Recall that the explicit forms for the \textit{exponential} mapping for $(\bx,\bv)\in T\mathbb{S}^1$ and the length-minimizing geodesic curve $\gamma$ on $t\in [0,1]$ are given as 
	\[
	\exp_{\bx}\bv =\cos(\|\bv\|)\bx+\frac{\sin(\|\bv\|)}{\|\bv\|}\bv, \quad \gamma(t):=\cos(\|\bv\|t)\bx+\frac{\sin(\|\bv\|t)}{\|\bv\|}\bv.  
	\]
	Thus, one has
	\begin{align}
	\begin{aligned} \label{AA-0}
		&\log_{\bx}\hat{\bx}:=\bv\neq 0 \iff \exp_{\bx}\bv=\hat{\bx} \iff \hat{\bx}=\cos(\|\bv\|)\bx+\frac{\sin(\|\bv\|)}{\|\bv\|}\bv \\
		& \hspace{1cm} \iff \bv=\|\bv\|\cdot\frac{\hat{\bx}-\cos(\|\bv\|)\bx}{\sin(\|\bv\|)} \iff \bv=d(\bx,{\hat{\bx}})\cdot\frac{\hat{\bx}-\langle \bx,\hat{\bx}\rangle \bx}{\sqrt{1-\langle \bx,\hat{\bx}\rangle^2}}.
	\end{aligned}
	\end{align} 
	
\subsection{A formal derivation of a bonding force} First, note that exponential and lograithmn maps can be defined only when an injectivity radius is well-defined, and the injectivity radius of the unit-circle is $\pi$. We set 
	\begin{equation} \label{AA-1}
	 \bx_i:=(\cos\theta_i,\sin\theta_i), \quad \bv_i:=\dot{\theta}_i(-\sin\theta_i,\cos\theta_i).
	 \end{equation} 
	 \begin{lemma} \label{LB-1}
	Suppose that the time-dependent ensemble $\{\theta_i(t) \}$ satisfies
	\begin{equation}  \label{AA-1-1}
	|\theta_j-\theta_i|<\pi, \quad t \geq 0.
	\end{equation}
	Then, the following estimates hold.
	\begin{align*}
	\begin{aligned}
	& (i)~d_{ij}=|\theta_j-\theta_i|, \quad \frac{\log_{\bx_i}\bx_j}{d_{ji}} =\mbox{sgn}(\theta_j-\theta_i) (-\sin\theta_i,\cos\theta_i). \\
	& (ii)~\Big \langle \bv_i,  \frac{\log_{\bx_i}\bx_j}{d_{ji}}   \Big \rangle +  \Big \langle \bv_j,  \frac{\log_{\bx_j}\bx_i}{d_{ji}}   \Big \rangle = -( \dot{\theta}_j  -\dot{\theta}_i)  \mbox{sgn}(\theta_j-\theta_i).
	\end{aligned}
	\end{align*}
	  \end{lemma}
	 \begin{proof} (i)~We apply \eqref{AA-0} with $\bx = \bx_i, \quad \hat{\bx} = \bx_j$ and use \eqref{AA-1} to find
	 \begin{align*}
	 \begin{aligned}
	\frac{\log_{\bx_i} {\bx}_j}{d_{ji}} &= \frac{{\bx}_j-\langle \bx_i,{\bx}_j\rangle \bx_i}{\sqrt{1-\langle \bx_i, {\bx}_j \rangle^2}} =  \frac{\sin(\theta_j - \theta_i)}{|\sin(\theta_j - \theta_i)|} (-\sin \theta_i, \cos \theta_i) \\
	 & = \mbox{sgn}(\theta_j - \theta_i)  (-\sin \theta_i, \cos \theta_i). 
	  \end{aligned}
	  \end{align*}
	  where we used \eqref{AA-1-1} in the last equality.  
	  
	  \vspace{0.2cm}
	  
	  \noindent (ii)~We use the result of (i) and $\eqref{AA-1}_2$ to find
	  \begin{align}
	  \begin{aligned} \label{AA-1-2}
	 \Big \langle \bv_i,  \frac{\log_{\bx_i}\bx_j}{d_{ji}}   \Big \rangle &= \Big \langle \dot{\theta}_i(-\sin\theta_i,\cos\theta_i),   \mbox{sgn}(\theta_j-\theta_i) (-\sin\theta_i,\cos\theta_i)   \Big   \rangle \\
	   &=  \dot{\theta}_i  \mbox{sgn}(\theta_j-\theta_i)  \Big \langle (-\sin\theta_i,\cos\theta_i),  (-\sin\theta_i,\cos\theta_i)    \Big  \rangle \\
	   &= \dot{\theta}_i  \mbox{sgn}(\theta_j-\theta_i).
	   \end{aligned}
	  \end{align}
	  By symmetry, one has 
	  \begin{equation} \label{AA-1-3}
	   \Big \langle \bv_j,  \frac{\log_{\bx_j}\bx_i}{d_{ji}}   \Big \rangle = \dot{\theta}_j  \mbox{sgn}(\theta_i-\theta_j) = -\dot{\theta}_j  \mbox{sgn}(\theta_j-\theta_i). 
	   \end{equation}
	  Finally, we combine \eqref{AA-1-2} and \eqref{AA-1-3} to get the desired estimate. 
	 \end{proof}
	As in the previous appendix, we set 
	 \[ e_{ij} = d_{ij} - \theta_{ij}^{\infty}, \quad v_{i,j} := - \Big \langle \bv_i,    \frac{\log_{\bx_i}\bx_j}{d_{ji}} \Big \rangle, \quad v_{j,i} := - \Big \langle \bv_j,  \frac{\log_{\bx_j}\bx_i}{d_{ji}} \Big  \rangle.  \]
	Then, one has 
	\[ {\dot e}_{ij} =  v_{i,j} + v_{j,i} =  - \Big \langle \bv_i,    \frac{\log_{\bx_i}\bx_j}{d_{ji}} \Big \rangle- \Big \langle \bv_j,  \frac{\log_{\bx_j}\bx_i}{d_{ji}} \Big  \rangle.   \]
	 Now, we set 
	 \[  f_{ji}  := - \kappa_1 {\dot e}_{ij} - \kappa_2 e_{ij} = \kappa_1 \Big( \Big \langle \bv_i,    \frac{\log_{\bx_i}\bx_j}{d_{ji}} \Big \rangle + \Big \langle \bv_j,  \frac{\log_{\bx_j}\bx_i}{d_{ji}} \Big  \rangle  \Big)  -\kappa_2 ( d_{ij} - \theta_{ij}^{\infty}).   \]
	 Now, we use Lemma \ref{LB-1}  to see
	 \[ f_{ji} = -\kappa_1 ( \dot{\theta}_j  -\dot{\theta}_i)  \mbox{sgn}(\theta_j-\theta_i)  -\kappa_2 ( d_{ij} - \theta_{ij}^{\infty}), \quad   \bn_{ij} = \mbox{sgn}(\theta_i - \theta_j),  \]
	These yield
	 \begin{equation} \label{AA-1-4}
	 \bbf_{ji} = f_{ji} \bn_{ij} = \kappa_1 ( \omega_j  -\omega_i)  + \kappa_2 ( d_{ij} - \theta_{ij}^{\infty})  \mbox{sgn}(\theta_j - \theta_i).  
	 \end{equation}
	 Finally, we combine \eqref{AA-0-0} and \eqref{AA-1-4} to derive  \eqref{Ku-SB}.


\end{document}